\documentclass[11pt]{amsart}
\usepackage[top=1in, bottom=1in, left=1in, right=1in]{geometry}
\usepackage[utf8]{inputenc}
\usepackage{etoolbox}
\usepackage{xcolor}
\usepackage{graphicx}
\usepackage{overpic}
\usepackage{cprotect}
\graphicspath{ {./figures/} }

\usepackage{caption}
\usepackage{subcaption}
\usepackage{amsmath, amsfonts, mathtools}
\usepackage{amssymb}
\usepackage{amsthm}
\usepackage{enumerate}
\usepackage{enumitem}
\usepackage{tikz}

\numberwithin{equation}{section}
\theoremstyle{definition}
\newtheorem{theorem}{Theorem}[section]
\newtheorem{proposition}{Proposition}[section]
\newtheorem{lemma}{Lemma}[section]
\newtheorem{corollary}{Corollary}[section]
\newtheorem{definition}{Definition}[section]
\newtheorem{remark}{Remark}[section]

\usepackage{hyperref}

\usepackage[T1]{fontenc}
\usepackage[final]{microtype}

\newcommand{\RR}{\mathbb{R}}
\newcommand{\CC}{\mathbb{C}}

\newcommand{\Aa}{\mathcal{A}}
\newcommand{\Cc}{\mathcal{C}}
\newcommand{\Ee}{\mathcal{E}}
\newcommand{\Ff}{\mathcal{F}}
\newcommand{\Gg}{\mathcal{G}}
\newcommand{\Hh}{\mathcal{H}}
\newcommand{\Ii}{\mathcal{I}}
\newcommand{\Jj}{\mathcal{J}}
\newcommand{\Ll}{\mathcal{L}}

\newcommand{\Oo}{\mathcal{O}}

\newcommand{\Qq}{\mathcal{Q}}
\newcommand{\Rr}{\mathcal{R}}
\newcommand{\Ss}{\mathcal{S}}
\newcommand{\Tt}{\mathcal{T}}

\renewcommand{\a}{\alpha}
\renewcommand{\b}{\beta}
\newcommand{\g}{\gamma}
\newcommand{\eps}{\epsilon}
\renewcommand{\i}{\iota}
\newcommand{\s}{\sigma}
\renewcommand{\t}{\theta}
\newcommand{\z}{\zeta}
\newcommand{\x}{\chi}

\renewcommand{\d}{\text{d}}

\newcommand{\p}{\partial}

\newcommand{\ip}[2]{\langle #1, #2 \rangle}
\newcommand{\norm}[1]{\lVert#1\rVert}

\newcommand{\Tr}{\text{Tr}}

\newcommand{\M}{M}
\newcommand{\Ms}{M_0}
\newcommand{\Me}{M_T}
\newcommand{\Mext}{\tilde{M}}
\newcommand{\Omegaext}{\tilde{\Omega}}

\newcommand{\gL}{\bar{g}}
\newcommand{\gextt}{\tilde{g}}
\newcommand{\gLextt}{\tilde{\bar{g}}}
\newcommand{\gext}{g}
\newcommand{\gLext}{\bar{g}}

\newcommand{\zz}{s}
\newcommand{\pp}{\tau}
\newcommand{\qq}{\rho}
\newcommand{\thetaext}{\tilde{\theta}}
\newcommand{\gammaext}{\tilde{\gamma}}

\renewcommand{\k}{k}
\newcommand{\ind}{\ell}

\newcommand{\epsf}{\eps_0}
\newcommand{\epsrand}{\kappa}
\newcommand{\C}{A}
\newcommand{\D}{B}
\newcommand{\bdf}{\omega}
\newcommand{\folf}{\lambda}
\newcommand{\amp}{a}

\renewcommand{\v}{\eta}
\newcommand{\ff}{h}

\title[Stability of the non-diffusive Westervelt inverse problem]{Stable determination of the nonlinear parameter in the non-diffusive Westervelt equation from the Dirichlet-to-Neumann map}
\author{Mike Wendels}

\begin{document}

\begin{abstract}
    The Westervelt equation models the propagation of nonlinear acoustic waves in a regime well-suited for applications such as medical ultrasound imaging. In this work, we prove that the nonlinear parameter, as well as the sound speed, can be stably recovered from the Dirichlet-to-Neumann map associated with the non-diffusive Westervelt equation in (1+3)-dimensions. This result is essential for the feasibility of reconstruction methods. The Dirichlet-to-Neumann map encodes boundary measurements by associating a prescribed pressure profile on the boundary with the resulting pressure fluctuations. We prove stability provided the sound speed is a priori known to be close to a reference sound speed and under certain geometrical conditions. We also verify the result through numerical experiments.
\end{abstract}

\maketitle

\section{Introduction} \label{sec:intro}

Ultrasound imaging is widely applied in fields such as medical imaging. One approach to image the medium of interest is by recovering the sound speed. In the scenario where nonlinear interactions of the acoustic waves with the underlying medium are significant, an alternative method is to recover the nonlinear acoustic parameter $B/A$. In the context of medical imaging, this nonlinear parameter has been shown to be significant and tissue-dependent (e.g., \cite{bjorno1986characterization}). Therefore, it can be used for tomography \cite{ichida1983imaging}. Moreover, tomography through the recovery of the nonlinear parameter, in addition to recovery of the sound speed alone, could improve the quality of medical ultrasound imaging.

A commonly used equation in the context of high intensity ultrasound is the Westervelt equation. This equation, which models the pressure fluctuation $u(t,x)$ to a reference pressure field, is given by
\begin{align*}
    \Box_c u(t,x) - \b(x) \p_{t}^2 u^2(t,x) - \eta c^{-2}(x) \Delta \p_t u(t,x) = 0,
\end{align*}
with $c>0$ the sound speed, $\Box_c := c^{-2} \p_t^2 - \Delta$ the acoustic wave operator, $\eta \geq 0$ the diffusivity of sound, and $\b := \left(1 + \frac{B}{2A} \right) \rho_0^{-1} c^{-4}$ the (modified) nonlinear coefficient where $\rho_0>0$ is the mean density field. This equation is an approximation to the Kuznetsov equation in the regime where the nonlinearity is weak and the dissipation is small, and the Kuznetsov equation in turn describes the pressure fluctuation $u$ derived from the Navier-Stokes equations by the decomposition of the pressure, density and velocity in a mean and fluctuative part. For a detailed discussion of these derivations, see, e.g., \cite{hamilton1998nonlinear}.

In this work, we will consider the non-diffusive Westervelt equation, which corresponds to setting $\eta = 0$ and thus to the case where the nonlinear effects dominate the diffusive effects of the medium of interest. Let $\Omega \subset \RR^{3}$ be a compact, connected domain with smooth boundary $\p \Omega$. The inverse problem considered in this work involves the recovery of the nonlinear parameter $\b$, together with the sound speed $c$ in case unknown, in the domain $\Omega$ by emitting acoustic waves into the domain from the boundary and measuring the induced pressure fluctuations at the boundary. These measurements can be described through the Dirichlet-to-Neumann (DN) map.

A higher-order linearization method will be used to analyze the inverse problem. This method was first considered to solve nonlinear hyperbolic inverse problems in \cite{kurylev2018inverse} and has been used to solve many other inverse problems since (e.g., \cite{feizmohammadi2021inverse}, \cite{hintz2022inverse}, \cite{lai2021reconstruction}, \cite{uhlmann2022inverse}, \cite{uhlmann2023inverse}). The idea behind the method is to linearize the DN map in order to obtain information about the desired parameters through linear problems. Due to the nature of the nonlinearity in the Westervelt equation, the second-order linearization of the DN map will be used to include $\b$ in the source term of a linear equation. The first-order linearization of the DN map will allow for the recovery of the sound speed $c$.

Given $c$, it is proven in \cite{acosta2022nonlinear} that the recovery of $\b$ is unique provided the underlying manifold $(\Omega,g)$ with acoustic metric $g_{ij}(x) = c^{-2}(x) \delta_{ij}$ is either simple or satisfies the foliation condition of \cite{uhlmann2016inverse}. It is moreover known that $\b$ can be uniquely recovered from the leading-order Burger's-type behavior of the non-diffusive Westervelt equation \cite{eptaminitakis2024weakly}. In \cite{kaltenbacher2021identification}, \cite{kaltenbacher2022inverse}, unique recovery of $\b$ given $c$ has been proven for the diffusive Westervelt equation with classical strong damping and for several fractional damping models. In \cite{kaltenbacher2023simultaneous}, unique recovery of both $\b$ and $c$ from interior measurements has been proven for the diffusive Westervelt equation, together with a reconstruction procedure.

In this work, we will prove that the recovery of $\b$ and $c$ from the DN map is stable for the non-diffusive Westervelt equation provided $c$ is a priori known to be close to a reference sound speed $c_0$. Moreover, this result requires sufficient regularity and a priori boundedness of the unknown parameters, and requires the foliation condition of \cite{uhlmann2016inverse} to be satisfied. The latter condition is satisfied, for example, when the reference sound speed $c_0$ is radial and satisfies the Herglotz condition $\p_r(r c_0^{-1}(r)) > 0$. The importance of this stability result lies in the fact that reconstruction methods have no chance of succeeding without the inverse problem being stable.

\subsection{Setting}
As mentioned above, we consider a compact, connected domain $\Omega \subset \RR^3$ with smooth boundary $\p \Omega$. Moreover, we will take the sound speed $c > 0$ and the nonlinearity $\b$ to be smooth and to depend only on the spatial variable. We will study the non-diffusive Westervelt equation given by
\begin{align} \label{eq:wv}
    \begin{dcases}
        \Box_c u - \b \p_t^2 u^2= 0, &\text{ in } \M, \\
        u = f, &\text{ on } \Sigma, \\
        u = \p_t u \equiv 0, &\text{ on } \Ms,
    \end{dcases}
\end{align}
where we denote
\begin{align*}
    \M := [0,T] \times \Omega, \quad \Sigma := [0,T] \times \p \Omega, \quad \M_t := \{t\} \times \Omega,
\end{align*}
for ease of notation, for some $T > 0$. Throughout this work, we fix $\Omega, T$. Furthermore, the notation $C_1 \lesssim C_2$ is used to denote that $C_1 \leq C C_2$ for some $C>0$ independent of $C_1, C_2$, and Latin indices run over $1,2,3$ while Greek indices run over $0,1,2,3$ regarding indices that are unspecified.

To discuss the well-posedness of equation \eqref{eq:wv}, we define the energy space 
\begin{align*}
    E^q(M) := \bigcap_{k=0}^q \Cc^k([0,T]; H^{q-k}(\Omega))
\end{align*}
for $q \geq 0$, which is equipped with the norm
\begin{align*}
    \norm{u}_{E^q(M)} := \sup_{t \in [0,T]} \sum_{k=0}^q \norm{\p_t^k u(t,\cdot)}_{H^{q-k}(\Omega)}.
\end{align*}
Since $H^q(M)$ only is an algebra for $q > \frac{3}{2}$ (see, e.g., \cite{choquet2009general}[Appendix I, Proposition 2.3]), we note that $E^q(M)$ is an algebra if $q \geq 2$, causing the norm estimate $\norm{uv}_{E^q(M)} \leq C_{q,\Omega,T} \norm{u}_{E^q(M)} \norm{v}_{E^q(M)}$ to hold for any $u,v \in E^q(M)$ and some $C_{q,\Omega,T} > 0$ dependent on $q,\Omega,T$ in that case.

From \cite{acosta2022nonlinear}, equation \eqref{eq:wv} is known to have a unique solution $u \in E^m(M), \, m \geq 5$ with estimate $\norm{u}_{E^{m}(M)}  \leq \tilde{C}_{m,c,\b,\Omega,T} \norm{f}_{\Cc^{m+1}(\Sigma)}$ for some $\tilde{C}_{m,c,\b,\Omega,T} > 0$ dependent on $m,c,\b, \Omega, T$, provided $c,\b$ are smooth with $c > 0$, and 
\begin{align} \label{eq:fcond_C}
    f \in \Cc^{m+1}(\Sigma), \, \norm{f}_{\Cc^{m+1}(\Sigma)} \leq \epsf, \, \p_t^k f|_{t=0} = 0, \quad \, k \leq m+1,
\end{align}
for some $\epsf > 0$ small enough. We also refer to \cite{kaltenbacher2011well} regarding well-posedness of the more general diffusive Westervelt equation. Due to the nature of the DN map considered in this work, we translate condition \eqref{eq:fcond_C} into a condition in terms of Sobolev spaces through the Sobolev embedding theorem, i.e., we require that
\begin{align*}
    f \in H_{0,\eps_0}^{s+3}(\Sigma) := \{ f \in H_0^{s+3}(\Sigma) : \norm{f}_{H^{s+3}(\Sigma)} \leq \epsf \}, \quad s > m-\tfrac{1}{2} \geq \tfrac{9}{2},
\end{align*}
for some $\eps_0 > 0$ small enough to guarantee well-posedness of equation \eqref{eq:wv}, where $H_{r}^s(\Sigma) := \{ f \in H^{s}(\Sigma) : \p_t^\k f|_{t=r} = 0, \, \k < s-\frac{3}{2} \}$. Hence, $f \in H_{0,\eps_0}^{s+3}(\Sigma)$ combined with $c>0,\b$ being smooth implies a unique solution $u \in E^s(M)$ to equation \eqref{eq:wv} with estimate
\begin{align*}
    \norm{u}_{E^{s}(M)}  \leq C_{s,c,\b,\Omega,T} \norm{f}_{H^{s+3}(\Sigma)}.
\end{align*}
for some $C_{s,c,\b,\Omega,T} > 0$ dependent on $s,c,\b, \Omega, T$. In the remainder of this work, we will fix the regularity constant $s$, and we will assume that $\eps_0 > 0$ is small enough for well-posedness.

Measurements to recover $c,\b$ can be made through the Dirichlet-to-Neumann (DN) map $f \mapsto \p_{\nu} u|_{\Sigma}$, where $\nu(p)$ is the outer unit normal vector to $p \in \p \Omega$ with respect to the Euclidean metric. To ensure injectivity of the DN map, we restrict the measurements through the DN map to only include those for boundary profiles that guarantee a unique solution, i.e., we define the DN map as 
\begin{align*}
    \Lambda: H_{0,\eps_0}^{s+3}(\Sigma) \rightarrow H^{s-2}(\Sigma), \quad \Lambda f := \p_{\nu} u|_{\Sigma}.
\end{align*}
Consider the DN maps $\Lambda_\ind, \, \ind=1,2$ corresponding to $c=c_\ind, \b=\b_\ind$. As a consequence of the way the DN map is defined, there is a $\delta > 0$ such that
\begin{align} \label{eq:DNdelta}
    \norm{(\Lambda_1 - \Lambda_2)(f)}_{H^{s - 2}(\Sigma)} \leq \delta,
\end{align}
provided $f \in H_{0,\eps_0}^{s+3}(\Sigma)$ and $c_\ind > 0, \b_\ind, \, \ind=1,2$ are smooth. The stability results in this work will therefore be in terms of this $\delta$.

\subsection{Main result}
As mentioned earlier, the stated results require the foliation condition of \cite{uhlmann2016inverse} to be satisfied, which will be given in Definition \ref{def:fol}. We note again that this condition is satisfied in the acoustic setting in the case of an isotropic radial sound speed $c(r)$ satisfying the Herglotz condition $\p_r(r c^{-1}(r)) > 0$, for example. 
Moreover, we require $T$ to be large enough to recover the desired parameters, as we will specify below.

We first consider a special case of the main result, which concerns the stable recovery of $\b$ from the DN map in the case where $c$ is a priori known. To this end, we set $T > \text{diam}_g(\Omega)$, where $\text{diam}_{g}(\Omega)$ denotes the supremum of the lengths of all unit-speed geodesics in $\Omega$ under the acoustic metric $g_{ij} = c^{-2} \delta_{ij}$. 

\begin{proposition} \label{thm:stabb1}
Let $\Omega \subset \RR^3$ be a compact, connected domain with smooth boundary $\p \Omega$, and let $\delta$ be such that inequality \eqref{eq:DNdelta} holds for all $f \in H_{0,\eps_0}^{s+3}(\Sigma)$.
Suppose that $c, \b_\ind \in \Cc^{\infty}(\Omega), \, c > 0$ with $\norm{\b_\ind}_{H^l(\Omega)} \leq C_\b$ for some $l \geq s+1, \, C_\b > 0$.
If $(\Omega,g), \, g_{ij} = c^{-2}\delta_{ij}$ satisfies the foliation condition as in Definition \ref{def:fol}, then
\begin{align*}
    \norm{\b_1 - \b_2}_{H^q(\Omega)} \lesssim \delta^\mu
\end{align*}
for any $q < l$, where $0 < \mu < 1$ is dependent on $s,l,q$.
\end{proposition}

The main result concerning stable recovery of both $c$ and $\b$ from the DN map is a generalization of Proposition \ref{thm:stabb1} by additionally demanding that the sound speed is a priori known to be close to some reference sound speed $c_0$. To this end, we set $T > \text{diam}_{g_0}(\Omega)$, with $(g_0)_{ij} = c_0^{-2} \delta_{ij}$.

\begin{theorem} \label{thm:stabb2}
Let $\Omega \subset \RR^3$ be a compact, connected domain with smooth boundary $\p \Omega$, and let $\delta$ be such that inequality \eqref{eq:DNdelta} holds for all $f \in H_{0,\eps_0}^{s+3}(\Sigma)$. Suppose $c_\ind \in \Cc^{\infty}(\Omega), \, c_\ind>0, \, \ind=1,2$ satisfy
\begin{align*}
    \norm{c_\ind - c_0}_{\Cc(\Omega)} \leq \kappa, \quad \norm{c_\ind}_{\Cc^k(\Omega)} \leq C_c, 
\end{align*}
for some $\kappa, C_c > 0$ and $c_0 \in \Cc^{k}(\Omega), \, k \gg 1, \, c_0 > 0$ such that $(\Omega, g_0),\, (g_0)_{ij} = c_0^{-2} \delta_{ij}$ satisfies the foliation condition as in Definition \ref{def:fol}. Moreover, suppose that $\b_\ind \in \Cc^{\infty}(\Omega)$ with $\norm{\b_\ind}_{H^l(\Omega)} \leq C_\b$ for some $l \geq s+1, \, C_\b > 0$. Then
\begin{align} \label{eq:stab_c_Lambda}
    \norm{c_1 - c_2}_{\Cc^2(\Omega)} \lesssim \delta^\mu, \quad \norm{\b_1 - \b_2}_{H^q(\Omega)} \lesssim \delta^{\tilde{\mu}},
\end{align} 
for any $q < l$, with $0 < \mu < 1$ dependent on $k$, and $0 < \tilde{\mu} < 1$ dependent on $k,s,l,q$.
\end{theorem} 

The outline of this work is as follows. In Section \ref{sec:preliminaries}, we discuss the foliation condition and estimates for the acoustic wave equation which will be used throughout this work. In Section \ref{sec:hol}, we discuss the higher-order linearization method, as the first- and second-order linearization of the DN map will allow for the recovery of respectively $c$ and $\b$. Moreover, we will derive an Alessandrini-type identity for $\b$. In Section \ref{sec:GB}, we construct Gaussian beam solutions for the linear terms appearing in the Alessandrini-type identity for $\b$, which allow us to recover a weighted geodesic ray transform of $\b$ from the DN map. In Section \ref{sec:thm1} and \ref{sec:thm2} we then prove Proposition \ref{thm:stabb1} and Theorem \ref{thm:stabb2}, respectively. Section \ref{sec:numerical_experiments} concludes with supporting numerical experiments.

\subsection*{Acknowledgements.} The author would like to thank Gunther Uhlmann for his guidance, support, and patience. The author would also like to thank Yuchao Yi for helpful discussions and comments. This work is partially supported by the NSF.

\section{Preliminaries} \label{sec:preliminaries}

\subsection{Foliation condition}
Throughout this work, let $(\Omega,g)$ be a compact, connected Riemannian manifold with smooth boundary $\p \Omega$ endowed with the acoustic metric $g_{ij} = c^{-2} \delta_{ij}$. Recall that $(\Omega,g)$ is called non-trapping if $\text{diam}_g(\Omega) < \infty$, and that a manifold is called simple if it is non-trapping, has a strictly convex boundary with respect to the metric, and has no conjugate points. We also recall that the scattering relation is defined as
\begin{align*}
    \Ll: S_-^* \p \Omega \rightarrow S_+^* \p \Omega, \quad \Ll(p,\xi') := (q,\eta'),
\end{align*}
taking a point $(p,\xi') \in S_-^* \p \Omega$ along the maximal unit-speed geodesic $\g$ emerging from $(p,\xi')$ in the metric $g$ to the corresponding point $(q,\eta') \in S_+^* \p \Omega$ where $\g$ intersects again with the boundary, where $S_{\pm}^* \p \Omega := \{ (p,\xi') \in S^* \p \Omega : \mp \langle \xi', \nu(p)\rangle_g < 0\}$. Furthermore, the boundary distance function $d: \p \Omega \times \p \Omega \rightarrow \RR$ describes the length $d(p,q)$ of the minimizing geodesic connecting $p$ and $q$.

The foliation condition was introduced in \cite{uhlmann2016inverse}, and implies that a manifold can be foliated by strictly convex hypersurfaces with respect to the metric. 

\begin{definition}[\cite{stefanov2016boundary}] \label{def:fol}
    If there exists a smooth function $\folf: \Omega \rightarrow [0,T)$ which level sets $\Gamma_t := \folf^{-1}(t), \, t \in [0,T)$ are strictly convex viewed from $\folf^{-1}((0,t))$ with respect to $g$, $\d \folf \neq 0$ on each level set, $\Gamma_0 = \p \Omega$, and $\Omega \backslash \cup_{t \in [0,T)} \Gamma_t$ has an empty interior, then $(\Omega,g)$ is said to satisfy the foliation condition.
\end{definition}

We emphasize that the foliation condition generalizes the notion of simpilicity by allowing for conjugate points, where we note that the foliation condition implies $(\Omega,g)$ being non-trapping \cite{stefanov2016boundary}[Proposition 5.1] and $\p \Omega$ being strictly convex with respect to the metric. Because the foliation condition allows for conjugate points, geodesics cannot be parameterized uniquely anymore by their starting point and endpoint at the boundary as was the case under simplicity. However, they can still be uniquely parameterized in $S_-^* \p \Omega$ as geodesic flow is continuous with respect to the metric (see, e.g., Lemma \ref{prop:stab_geod_flow_0}).

The recovery of $\b$ from the DN map as considered in this work involves the inversion of a weighted geodesic ray transform $I \b$. This inversion is possible under the foliation condition as proven in \cite{uhlmann2016inverse}. We first discuss the local result, for which we take $p \in \p \Omega$ so that the boundary is locally strictly convex with respect to the metric. We denote the corresponding strictly convex boundary segment by $\Gamma_p$, and we note that in this local setting the condition $\Gamma_0 = \p \Omega$ in Definition \ref{def:fol} can be relaxed to $\Gamma_p \subset \Gamma_0 \cap \p \Omega$, and that the interior of $\Omega \backslash \cup_{t \in [0,T)} \Gamma_t$ does not need to be empty. Moreover, let $\folf$ be as in Definition \ref{def:fol}, let $\bdf$ be a boundary defining function (i.e., $\bdf \in \Cc^{\infty}(\Omega)$, $\bdf > 0$ in $\Omega$ and vanishing non-degenerately at $\p \Omega$), and consider an open region $O_p(\epsrand) := \{\folf > -\epsrand\} \cap \{ \bdf > 0 \}$ around $p$ for some $\epsrand > 0$ small enough such that $\p \bar{O}_p \cap \p \Omega \subset \Gamma_p$. A sketch of the setting is given in figure \ref{fig:lens_shaped_region}.

We can ensure that $O_p$ does not contain conjugate points by possibly further decreasing $\epsrand$, as the foliation condition ensures that conjugate points can be avoided locally due to strict convexity of the level sets $\Gamma_t$ with respect to the metric. The weighted geodesic ray transform can then be inverted along the geodesics that are nearly tangential to $p$, and thus in $O_p(\kappa)$ for $\kappa$ small. The inversion will take the weighted geodesic ray transform on $O_p$ to the space $H_{\mathsf{F}_p}^q(O_p)$ for some $q \geq -1$, where
\begin{align*}
    H_{\mathsf{F}_p}^q(O_p) := \{ f \in H_{\text{loc}}^q (O_p) : e^{-\mathsf{F}_p/(\folf+\epsrand)} f \in H^q(O_p) \}
\end{align*}
with $\mathsf{F}_p>0$ some constant dependent on $O_p$.

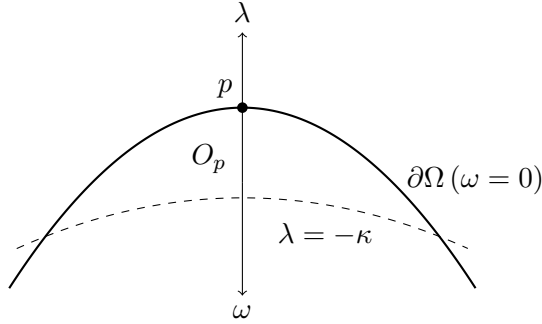
\begin{figure}[ht]
    \centering
    \begin{tikzpicture}[scale=2]

        \draw[->] (0,1) -- (0,2) node[above] {$\folf$};
        \draw[->] (0,1) -- (0,0.25) node[below] {$\bdf$};

        \draw[dashed,black,domain=-1.5:1.5,samples=100] plot(\x,{0.9-0.15*\x*\x}) node[below] at (0.55,0.8) {$\folf = -\epsrand$};

        \draw[thick,black,domain=-1.55:1.55,samples=100] plot(\x,{1.5-0.5*\x*\x}) node[above] at (1.55,0.85) {$\p \Omega \, (\bdf = 0)$};
        
        \fill[black] (0,1.5) circle (1pt) node[above left] {$p$};
        \fill[black] node[above left] at (-0.04,1) {$O_p$};
        
    \end{tikzpicture}
    \caption{Sketch of the setting of \cite{uhlmann2016inverse} to locally invert a weighted geodesic ray transform around some point $p \in \p \Omega$ for which the boundary is locally strictly convex with respect to the metric.}
    \label{fig:lens_shaped_region}
\end{figure}

Under the foliation condition, the above local result can be extended to the entire domain $\Omega$ through a layer stripping argument. First, the weighted geodesic ray transform can be locally inverted in a neighborhood of $\p \Omega$ as described above. This procedure can then be repeated by considering appropriate geodesics in $S_-^* \p \Omega$, allowing one to march inward into the domain while avoiding conjugate points. In this way, $\b$ can be recovered stably from $I \b|_{\Gg}$ on $\Omega$, where $\Gg$ denotes the family of geodesics considered in this construction. The inversion will then cause $\b \in H_\mathsf{F}^q(\Omega)$ for some $q \geq -1$, where
\begin{align*}
    H_\mathsf{F}^q(\Omega) := \{ f \in H^q(\Omega) : f|_{O_{p_i}} \in H_{\mathsf{F}_i}^q(O_{p_i}), \, \cup_{i=1}^{\infty} O_{p_i} = \Omega \}
\end{align*}
with $\mathsf{F} := \{ \mathsf{F}_{p_i} \}$ the weights for each local lens-shaped region in the used partition of $\Omega$. Moreover, this inversion is stable with estimate
\begin{align} \label{eq:stab_wgrt}
    \norm{\b}_{H_\mathsf{F}^{q-1}(\Omega)} \lesssim \norm{I \b|_{\Gg}}_{H^q(D)}, 
\end{align}
for any $q \geq 0$, where $D \subset S_-^* \p \Omega$ parameterizes the geodesics in $\Gg$. We refer to \cite{uhlmann2016inverse} for more details.

\subsection{Estimates for the acoustic wave equation}
Consider the Lorentzian manifold $(M,\bar{g}), \, \bar{g} = -\d t^2 + g$ corresponding to the Riemannian manifold $(\Omega,g), \, g = c^{-2} \delta_{ij}$, where $c \in \Cc^{\infty}(\Omega), c>0$ with $\norm{c}_{\Cc^k(\Omega)} \leq C_c$ for some $k \gg 1, \, C_c > 0$. Let $\Box_{\gL}$ be the d'Alembertian, i.e.,
\begin{align*}
    \Box_{\gL} := |\det \gL|^{-\frac 1 2} \p_\mu (|\det \gL|^{\frac 1 2} \gL^{\mu \nu} \p_\nu)
\end{align*}
in coordinates $(x^\mu) = (t, x^i)$, where $\p_\mu := \p_{x^\mu}$ and $\bar{g}^{\mu \nu} := (\bar{g}^{-1})_{\mu \nu}$. We then have the following results regarding the geometric wave equation, where the first result is based on \cite{kachalov2001inverse}[Theorem 2.45] combined with the argument of \cite{lassas2020uniqueness}[Corollary 2] with $\Delta$ replaced by the Laplace-Beltrami operator $\Delta_g$, and where the second result is an immediate consequence of \cite{kachalov2001inverse}[Corollary 2.36].

\begin{lemma}[\cite{kachalov2001inverse},\cite{lassas2020uniqueness}] \label{lem:estimate1}
Let $q \geq 0$, $\p_t^k F \in L^1([0,T]; H^{q-k}(\Omega)), \, k=0,1,\dots,q$, and $f \in H_0^{q+1}(\Sigma)$. Then the equation 
\begin{align} \label{eq:lwe}
    \begin{dcases}
    \Box_{\bar{g}} v = F, &\text{ in } \M, \\
    v = f, &\text{ on } \Sigma, \\
    v = \p_t v \equiv 0, &\text{ on } \Ms,
    \end{dcases}
\end{align}
has a unique solution $v \in E^{q+1}(M)$ with $\p_{\nu} v|_{\Sigma} \in H^q(\Sigma)$ and
\begin{align*}
    \norm{v}_{E^{q+1}(M)} + \norm{\p_{\nu} v|_{\Sigma}}_{H^q(\Sigma)} \leq C_{q,T} \left( \sum_{k=0}^q \norm{\p_t^k F}_{L^1([0,T];H^{q-k}(\Omega))} + \norm{f}_{H^{q+1}(\Sigma)} \right)
\end{align*}
for some $C_{q,T} > 0$ dependent on $q,T$.
\end{lemma}

\begin{lemma}[\cite{kachalov2001inverse}] \label{lem:estimate2}
Consider equation \eqref{eq:lwe} with $F \in L^1([0,T]; H^{-1}(\Omega))$ and $f \equiv 0$. Then there is a unique solution $v \in \Cc([0,T];L^2(\Omega)) \cap \Cc^1([0,T];H^{-1}(\Omega))$ with estimate
\begin{align*}
    \max_{t\in[0,T]} \left( \norm{v}_{L^2(\Omega)} + \norm{\p_t v}_{H^{-1}(\Omega)} \right) \leq C_T \norm{F}_{L^1([0,T];H^{-1}(\Omega))}
\end{align*}
for some $C_T > 0$ dependent on $T$. 
\end{lemma}

We will now translate these results in terms of the acoustic wave equation. We first define a scaled acoustic operator, i.e.,
\begin{align*}
    P u := -c^2 \Box_cu = (-\p_t^2 + c^2 \Delta) u
\end{align*}
and we observe that
\begin{align} \label{eq:P-BoxgL}
    Pu = \Box_{\gL} u + \ip{\d (\log c)}{\d u}_{\gL}.
\end{align}
Since $P$ is a first-order perturbation of $\Box_{\gL}$, their principal symbols are equal and the above results can still be observed to hold for $\Box_{\gL}$ replaced by $P$, with only a change of constants in the estimates. This is due to the fact that Gårding's inequality (see, e.g., \cite{kachalov2001inverse}[Theorem 2.22]) which is underlying to \cite{kachalov2001inverse}[Corollary 2.36, Theorem 2.45] is unaltered up to the constant involved. Hence, the results of Lemma \ref{lem:estimate1}, Lemma \ref{lem:estimate2} are still observed to hold with $\Box_{\gL}$ replaced by $P$. As $\Box_c$ is equal to $P$ up to scaling, the same holds for $\Box_c$. 

\begin{corollary} \label{cor:estimate1}
    The results of Lemma \ref{lem:estimate1}, Lemma \ref{lem:estimate2} hold for $\Box_{\gL}$ replaced by $P$ or $\Box_c$.
\end{corollary}

\section{Higher-order linearization method} \label{sec:hol}
In this section we discuss the first- and second-order linearization of the DN map, which will be an essential tool for the recovery of $c$ and $\b$. From the second-order linearization, we will derive an Alessandrini-type identity for $\b$.

\begin{lemma} \label{prop:hol}
Let $f \in H_{0}^{s+3}(\Sigma)$, and let $\eps > 0$ be small enough such that $\norm{\eps f}_{H^{s+3}(\Sigma)} \leq \epsf$. Suppose that $c,\b \in \Cc^{\infty}(\Omega), \, c > 0$ with $\norm{c}_{\Cc^{k}(\Omega)} \leq C_c$, $\norm{\b}_{H^{s+1}(\Omega)} \leq C_\b$ for some $k \gg 1, \, C_c, C_\b > 0$. Then there is a unique solution $u \in E^s(M)$ to
\begin{align*}
    \begin{dcases}
    \Box_c u - \b \p_t^2 u^2 = 0, &\text{ in } \M, \\
    u = \eps f, &\text{ on } \Sigma, \\
    u = \p_t u \equiv 0, &\text{ on } \Ms,
    \end{dcases}
\end{align*}
satisfying the estimate
\begin{align} \label{eq:est_u_epsf}
    \norm{u}_{E^{s}(M)} \lesssim \eps \norm{f}_{H^{s+3}(\Sigma)}.
\end{align}
Moreover, $u$ can be expanded as
\begin{align} \label{eq:exp}
    u = \eps v + \Qq = \eps v + \eps^2 w + \Rr, 
\end{align}
where $v$ solves
\begin{align} \label{eq:lin1}
    \begin{dcases}
    \Box_c v = 0, &\text{ in } \M, \\
    v = f, &\text{ on } \Sigma, \\
    v = \p_t v \equiv 0, &\text{ on } \Ms,
    \end{dcases}
\end{align}
with $v \in E^{s+3}(\M), \, \p_{\nu} v|_{\Sigma} \in H^{s+2}(\Sigma)$ and estimate 
\begin{align}\label{eq:est_v}
    \norm{v}_{E^{s+3}(\M)} + \norm{\p_{\nu} v|_{\Sigma}}_{H^{s+2}(\Sigma)}  &\lesssim \norm{f}_{H^{s+3}(\Sigma)},
\end{align}
where $w$ solves
\begin{align} \label{eq:lin2}
    \begin{dcases}
    \Box_c w - \b \p_t^2 v^2 = 0, &\text{ in } \M, \\
    w \equiv 0, &\text{ on } \Sigma, \\
    w = \p_t w \equiv 0, &\text{ on } \Ms,
    \end{dcases}
\end{align}
with $w \in E^{s+2}(\M), \, \p_{\nu} w|_{\Sigma} \in H^{s+1}(\Sigma)$ and estimate
\begin{align}\label{eq:est_w}
    \norm{w}_{E^{s+2}(\M)} + \norm{\p_{\nu} w|_\Sigma}_{H^{s+1}(\Sigma)} &\lesssim \norm{f}_{H^{s+3}(\Sigma)}^2,
\end{align}
and where the remainder terms $\Qq, \Rr$ satisfy 
\begin{align}
    \norm{\Qq}_{E^{s-1}(M)} + \norm{\p_{\nu} \Qq|_{\Sigma}}_{H^{s-2}(\Sigma)} &\lesssim \eps^2 \norm{f}_{H^{s+3}(\Sigma)}^2, \label{eq:Qest} \\
    \norm{\Rr}_{E^{s-1}(M)} + \norm{\p_{\nu} \Rr|_{\Sigma}}_{H^{s-2}(\Sigma)} &\lesssim \eps^3 \norm{f}_{H^{s+3}(\Sigma)}^3. \label{eq:Rest}
\end{align}
\end{lemma}

\begin{proof}
Uniqueness of the solution $u \in E^s(M)$ and the corresponding estimate \eqref{eq:est_u_epsf} follow from the discussion in Section \ref{sec:intro} combined with the a priori bounds on $c,\b$, where it can be observed from the proof of \cite{acosta2022nonlinear}[Theorem 2.1] (in particular, from the requirement that conditions (B1)-(B3) in \cite{dafermos1985energy} are satisfied) that $\norm{\b}_{H^{s+1}(\Omega)} \leq C_\b$ suffices. 

The fact that $v \in E^{s+3}(\M), \, \p_{\nu} v|_{\Sigma} \in H^{s+2}(\Sigma)$ with estimate \eqref{eq:est_v} immediately follows from Corollary \ref{cor:estimate1}. As a consequence, again using Corollary \ref{cor:estimate1}, we get that $w \in E^{s+2}(\M), \, \p_{\nu} w|_{\Sigma} \in H^{s+1}(\Sigma)$ with estimate
\begin{align*}
    \norm{w}_{E^{s+2}(M)} + \norm{\p_{\nu} w|_\Sigma}_{H^{s+1}(\M)} &\lesssim \norm{\Box_c w}_{E^{s+1}(M)} \\
    &\lesssim \norm{\b}_{H^{s+1}(\Omega)} \norm{v}_{E^{s+3}(M)}^2 \\ 
    &\lesssim \norm{f}_{H^{s+3}(\Sigma)}^2,
\end{align*}
where we recall that that $E^q(M), \, q \geq 2$ is an algebra.

Regarding expansion \eqref{eq:exp}, we note that $\Qq$ satisfies
\begin{align*}
    \begin{dcases}
    \Box_c \Qq - \b \p_t^2 u^2 = 0, &\text{ in } \M, \\
    \Qq \equiv 0, &\text{ on } \Sigma, \\
    \Qq = \p_t \Qq \equiv 0, &\text{ on } \Ms.
    \end{dcases}
\end{align*}
Using Corollary \ref{cor:estimate1} again, it follows that 
\begin{align*}
    \norm{\Qq}_{E^{s-1}(M)} + \norm{\p_{\nu} \Qq|_{\Sigma}}_{H^{s-2}(M)} &\lesssim \norm{\Box_c \Qq}_{E^{s-2}(M)} \\
    &\lesssim \norm{\b}_{H^{s-2}(\Omega)} \norm{u}_{E^{s}(M)}^2 \\
    &\lesssim \eps^2 \norm{f}_{H^{s+3}(\Sigma)}^2.
\end{align*}

The remainder term $\Rr$ satisfies
\begin{align*} 
    \begin{dcases}
    \Box_c \Rr = \b \p_t^2 (u^2 - (\eps v)^2) = \b \p_t^2 (\Qq \Ss), &\text{ in } \M, \\
    \Rr \equiv 0, &\text{ on } \Sigma, \\
    \Rr = \p_t \Rr \equiv 0, &\text{ on } \Ms,
    \end{dcases}
\end{align*}
where $\Ss := u + \eps v$, with $\norm{\Ss}_{E^s(M)} \lesssim \norm{u}_{E^s(M)} + \norm{\eps v}_{E^s(M)} \lesssim \eps \norm{f}_{H^{s+3}(\Sigma)}$ from estimates \eqref{eq:est_u_epsf}, \eqref{eq:est_v}.
As a consequence, we have that 
\begin{align*}
    \norm{\Rr}_{E^{s-1}(M)} + \norm{\p_{\nu} \Rr|_{\Sigma}}_{H^{s-2}(\Sigma)} &\lesssim \norm{\Box_c \Rr}_{E^{s-2}(M)} \\
    &\lesssim \norm{\b}_{H^{s-2}(\Omega)} \norm{\Qq}_{E^{s}(M)} \norm{\Ss}_{E^{s}(M)} \\
    &\lesssim \eps^3 \norm{f}_{H^{s+3}(\Sigma)}^3,
\end{align*}
from Corollary \ref{cor:estimate1} and the derived estimates for $\norm{\Qq}_{E^s(M)}, \, \norm{\Ss}_{E^s(M)}$. 
\end{proof}

We thus observe that the first- and second-order linearizations of the DN map are given by
\begin{align*}
    \p_{\eps} \Lambda(\eps f) = \p_{\nu} v|_{\Sigma}, \quad \p_{\eps}^2 \Lambda(\eps f) = \p_{\nu} w|_{\Sigma},
\end{align*}
respectively. We note that the first-order DN map allows for the recovery of $c$, and that the second-order DN map allows for the recovery of $\b$ due to the second-order nature of the nonlinearity in equation \eqref{eq:wv}. We also emphasize that the second-order linearization incorporates knowledge of the first-order linearization through equation \eqref{eq:lin2}. The following lemma allows us to recover the first- and second-order DN maps from the (full) DN map $\Lambda$ provided we appropriately control the remainder terms in Lemma \ref{prop:hol} through $f$.

\begin{lemma} \label{lem:DN12}
    Suppose the hypotheses of Lemma \ref{prop:hol} hold. Let $\Lambda_{\ind}$ be the DN map corresponding to $c=c_\ind, \b=\b_{\ind}$, and let $\delta$ be such that inequality \eqref{eq:DNdelta} holds for all $\eps f \in H_{0,\eps_0}^{s+3}(\Sigma)$. 
    Then
    \begin{align*}
        \norm{(\p_{\eps} \Lambda_{1} - \p_{\eps} \Lambda_{2})(\eps f)}_{H^{s-2}(\Sigma)} &\lesssim \eps^{-1} \delta + \eps \norm{f}_{H^{s+3}(\Sigma)}^2, \\ 
        \norm{(\p_{\eps}^2 \Lambda_{1} - \p_{\eps}^2 \Lambda_{2})(\eps f)}_{H^{s-2}(\Sigma)} &\lesssim \eps^{-2} \delta + \eps \norm{f}_{H^{s+3}(\Sigma)}^3,
    \end{align*}
    for any $f \in H_{0}^{s+3}(\Sigma)$ such that $\eps f \in H_{0,\eps_0}^{s+3}(\Sigma)$.
\end{lemma}

\begin{proof}
Let $u_{\ind,\eps f}, v_{\ind,\eps f}, w_{\ind,\eps f}, \Qq_{\ind,\eps f}, \Rr_{\ind,\eps f}$ be as defined in Lemma \ref{prop:hol} corresponding to equation \eqref{eq:wv} with $c=c_\ind, \b=\b_{\ind}, \, u|_{\Sigma} = \eps f$, and recall that $u_{\ind,\eps f} = \eps v_{\ind,\eps f} + \Qq_{\ind,\eps f} = \eps v_{\ind,\eps f} + \eps^2 w_{\ind,\eps f} + \Rr_{\ind,\eps f}$. 
We then readily obtain the estimate for the first-order DN map as
\begin{align*}
    \norm{(\p_{\eps} \Lambda_{1} - \p_{\eps} \Lambda_{2})(\eps f)}_{H^{s-2}(\Sigma)}
    &\leq \eps^{-1} \left( \norm{(\Lambda_1 - \Lambda_2)(\eps f)}_{H^{s-2}(\Sigma)} + \norm{\p_{\nu} (\Qq_{1,\eps f} - \Qq_{2,\eps f})|_{\Sigma}}_{H^{s-2}(\Sigma)} \right) \\
    &\lesssim \eps^{-1} \delta + \eps \norm{f}_{H^{s+3}(\Sigma)}^2.
\end{align*}
where we used estimate \eqref{eq:Qest}. Regarding the estimate for the second-order DN map, we define
\begin{align*}
    D_{\eps}^2 u_{\ind} &:= 2 \eps^{-2} (u_{\ind,\eps f} - 2 u_{\ind, \frac{\eps}{2}  f}), \\
    D_{\eps}^2 \Rr_{\ind} &:= 2 \eps^{-2} (\Rr_{\ind,\eps f} - 2 \Rr_{\ind,\frac{\eps}{2}  f}),
\end{align*}
and note that
\begin{align*}
    w = D_{\eps}^2 u - D_{\eps}^2 \Rr,
\end{align*}
as in, e.g., \cite{lassas2020uniqueness}. 
Therefore we get that
\begin{align*}
    \norm{(\p_{\eps}^2 \Lambda_{1} - \p_{\eps}^2 \Lambda_{2})(\eps f)}_{H^{s-2}(\Sigma)} 
    &\leq \norm{D_{\eps}^2 (\Lambda_1 - \Lambda_2)(\eps f)}_{H^{s-2}(\Sigma)} + \norm{\p_{\nu} (D_{\eps}^2 \Rr_1 - D_{\eps}^2 \Rr_2)|_{\Sigma}}_{H^{s-2}(\Sigma)} \\
    &\lesssim \eps^{-2} \delta + \eps \norm{f}_{H^{s+3}(\Sigma)}^3.
\end{align*}
from estimate \eqref{eq:Rest}.
\end{proof}

We will first focus on the recovery of $\b$ given $c$, i.e., proving Proposition \eqref{thm:stabb1}, for which we will derive an Alessandrini-type identity for $\b$. To this end, we also need a measurement function $\v$, which we define to be the solution to the following backward acoustic wave equation
\begin{align} \label{eq:lin0}
    \begin{dcases}
    \Box_c \v = 0, &\text{ in } \M, \\
    \v = \ff, &\text{ on } \Sigma, \\ 
    \v = \p_t \v \equiv 0, &\text{ on } \Me,
    \end{dcases}
\end{align}
with $\ff \in H_T^{s+3}(\Sigma)$.
From integration by parts, we then get that
\begin{align} \label{eq:ibp}
    \int_{\Sigma} (\p_{\nu} w)|_{\Sigma} \ff \, \d S \d t &= \int_\M \Delta w \v \, \d x \d t + \int_\M \nabla w \cdot \nabla \v \, \d x \d t \nonumber \\
    &= \int_\M \left( c^{-2} \p_t^2 w - \Box_c w \right) \v \, \d x \d t - \int_\M w \Delta \v \, \d x \d t \nonumber \\
    &= \int_\M w \Box_c \v \, \d x \d t - \int_{M} (\Box_c w) \v \, \d x \d t \nonumber \\
    &= - \int_{M} (\Box_c w) \v \, \d x \d t.
\end{align}
Hence, we get the following Alessandrini-type identity
\begin{align} \label{eq:wv_alessandrini}
    \int_{\Sigma} \p_{\eps}^2 \Lambda (\eps f) \ff \, \d S \d t 
    = \int_{M} \b \p_t v^2 \p_t \v \, \d x \d t.
\end{align}
In other words, the quantity $\int_{\M} \b \p_t v^2  \p_t \v \, \d x \d t$ can be obtained from knowledge of the second-order DN map. From this quantity $\b$ can be recovered by appropriately choosing $f,\ff$ such that the product $\p_t v^2 \p_t \v$ is dense in $L^1(\M)$.

\section{Construction of Gaussian beams for the linear acoustic wave equation} \label{sec:GB}

In this section, we construct approximate solutions for $v$ and $\v$ in the derived Alessandrini-type identity \eqref{eq:wv_alessandrini}. This is done in a geometric setting, by constructing solutions on the Lorentzian manifold $(M,\gL), \, \gL := -\d t^2 + g$. For the approximate solutions, Gaussian beams will be used. To account for the construction of the Gaussian beams up to the boundary $\Sigma$, the Gaussian beams are constructed on a slight extension of $(M,\gL)$ denoted by $(\Mext, \gLextt), \gLextt = -d t^2 + \gextt$, where $\gextt$ is a smooth, positive extension of $g$ by extending $c$ smoothly while ensuring $c > 0$. We denote the corresponding Riemannian extension by $(\Omegaext,\tilde{g})$. For ease of notation, we will drop the tilde denoting the extended metrics as this is clear from the context. Moreover, we will slightly abuse notation and denote $c(\pi_{\Omegaext}(q)),\b(\pi_{\Omegaext}(q)), \, q \in M$ by $c(q),\b(q)$, where $\pi_{\Omegaext}: \Mext \rightarrow \Omegaext$ is the projection down to $\Omegaext$. 

Gaussian beams are constructed in a tubular region around a null geodesic $\thetaext$ in $(\Mext,\gLext)$. We take $\thetaext$ to be the null geodesic corresponding to some unit-speed geodesic $\gammaext$ in $(\Omegaext,\gext)$, i.e.,
\begin{align*}
    \thetaext(t) = (t,\gammaext(t)), \quad t_- - \tilde{\qq} < t < t_+ + \tilde{\qq},
\end{align*}
with $t_\pm$ such that $(\gammaext(t_\pm), \dot{\gammaext}(t_\pm)) \in S_{\pm}^* \p \Omega$, and with $\tilde{\qq} > 0$ small enough such that $\thetaext \in \Mext$. The construction of the Gaussian beams will be discussed in the context of $P v = 0$ rather than for $\Box_c v = 0$, as $P$ is more naturally related to $\Box_{\bar{g}}$ through equation \eqref{eq:P-BoxgL}.

\subsection{Fermi coordinates}
It is natural to construct the Gaussian beams in Fermi coordinates $(z^\mu) = (z^0 = \zz,z')$, which are coordinates centered around the null geodesic $\thetaext$ of interest and for which the metric $\gL$ is locally flat around $\thetaext$. We introduce Fermi coordinates for the setting considered in this work, and we refer to, e.g., \cite{feizmohammadi2022recovery}, for a more general discussion.

Let $q$ be an arbitrary point on $\thetaext$, i.e., $q = \thetaext(t)$ for some $t_- - \tilde{\qq} < t < t_+ + \tilde{\qq}$, and define $e_0(q) := \dot{\thetaext}(t)$. Since $\thetaext$ is a null geodesic, $\ip{e_0(q)}{e_0(q)}_{\gL} = 0$ needs to hold, and we set
\begin{align*}
    e_0(q) = \frac{1}{\sqrt{2}}(1,c(q),0,0).
\end{align*}
By additionally setting
\begin{align*}
    e_1(q) = \frac{1}{\sqrt{2}}(-1,c(q),0,0), \quad e_2(q) = (0,0,c(q),0), \quad e_3(q) = (0,0,0,c(q)),
\end{align*}
we have created a pseudo-orthonormal basis $(e_\mu(q))$ on $T_q \tilde{M}$ as $\ip{e_\mu(q)}{e_\nu(q)}_{\gL}$ is only non-zero in the cases $\ip{e_0(q)}{e_1(q)}_{\gL} = \ip{e_k(q)}{e_k(q)}_{\gL} = 1, \, k=2,3$. Since $q$ was an arbitrary point on $\thetaext$, these local bases can be concatenated along $\thetaext$, giving the Fermi coordinate system around $\thetaext$. More precisely, let $p = \thetaext(t_-)$ be the reference point corresponding to $(s_-,0)$ in Fermi coordinates $(s,z')$. Then Fermi coordinates around $\thetaext$ are constructed by a parallel transport of the pseudo-orthonormal basis $(e_\mu(p))$ along $\thetaext$, i.e.,
\begin{align*}
    (s,z') := \Ff^{-1}(t,x),
\end{align*}
where $\Ff: \Mext \rightarrow \Mext $ is defined as
\begin{align*}
    \Ff(s, z') := \exp_{\thetaext(s)} (z^i E_i(s)),
\end{align*}
where $\exp_p: T_p \tilde{\M} \rightarrow \tilde{\M}$ is the exponential map on $\tilde{\M}$ at $p$, and where $E_\mu(s) \in T_{\thetaext(s)} \tilde{M}$ denotes the parallel transport of $e_{\mu}(p)$ along $\thetaext$ to $\thetaext(s)$. Since we know that $E_\mu(s) = e_\mu(\thetaext(s))$, we get that
\begin{align*}
    \d s(t,x) &= \frac{1}{\sqrt{2}} (\d t + c(x) \d x^1 ), \\
    \d z^1(t,x) &= \frac{1}{\sqrt{2}} (-\d t + c(x) \d x^1 ), \\
    \d z^k(t,x) &= c(x) \d x^k, \quad k=2,3,
\end{align*}
and thus we have that
\begin{align*}
    \d t(s,z') &= \frac{1}{\sqrt{2}} (\d s - \d z^1), \\
    \d x^1(s,z') &= \frac{1}{\sqrt{2}} c^{-1}(s,z') (\d s + \d z^1), \\
    \d x^k(s,z') &= c^{-1}(s,z') \d z^k, \quad k=2,3.
\end{align*}
Moreover, we observe that $\thetaext(s) = (s,0)$, and we note that the constructed Fermi coordinates are only well-defined in
\begin{align*}
    V_{\qq} := \{ (\zz, z') \in \Mext : \zz \in ( \zz_- - \qq, \zz_+ + \qq), \, \norm{z'} < \qq \},
\end{align*}
with $\qq > 0$ small enough to guarantee injectivity of $\Ff$ in $V_{\qq}$, where $s=s_+$ is such that $\thetaext(s_+) = \thetaext(t_+)$. Moreover, we assume $\qq>0$ is small enough such that $V_{\qq} \subset \Mext$, and we assume without loss of generality that $0 < \qq \leq 1$.

From the above construction, we immediately see that the metric $\bar{g}$ on $\tilde{\t}$ in Fermi coordinates is given by
\begin{align} \label{eq:gbar}
    \gLext(\zz,z')|_{\thetaext} &= 2 \d \zz \d z^1 + (\d z^2)^2 + (\d z^3)^2,
\end{align}
and that the first-order terms of the metric vanish around $\thetaext$ \cite{feizmohammadi2022recovery}[Lemma 1], i.e.,
\begin{align} \label{eq:gbar_prop}
    \frac{\p \gLext_{\mu \nu}}{\p z^\s} \bigg|_{\thetaext} &= 0. 
\end{align}
Moreover, since the metric is smooth, we get from equations \eqref{eq:gbar}, \eqref{eq:gbar_prop} that $\gL$ can be expanded around $\thetaext$ in $V_{\qq}$ as
\begin{align} \label{eq:gbarVrho}
\gL(s,z') = 2 \d z^0 \d z^1 + (\d z^2)^2 + (\d z^3)^2 + \sum_{m=2}^{\infty} (G_{\mu \nu}^m(s))_{i_1 i_2 \dots i_m} z^{i_1} z^{i_2} \cdots z^{i_m} \d z^\mu \d z^\nu,
\end{align}
with $G_{\mu \nu}^{m}(s)$ the coefficients of order $m$ of the metric. Since the construction of the Gaussian beam solutions for $P v = 0$ below is done in Fermi coordinates around $\thetaext$, we use the notation $\p_\mu := \p_{z^\mu}$ and drop the dependence of the metric $\bar{g}$ on $(s,z')$ throughout this discussion for ease of notation.

\subsection{WKB approximation}
We will construct a Gaussian beam solution for $P v = 0$ through the WKB ansatz, i.e.,
\begin{align} \label{eq:WKB}
    v_{\pp} = a_{\pp} e^{\i \pp \phi}, \quad \pp \geq 1,
\end{align}
with $\phi, a_{\pp} \in \Cc^{\infty}(\Mext)$, giving
\begin{align} \label{eq:P_WKB}
    P v_{\pp} &= ( -\pp^2 a_{\pp} \Hh \phi + \i \pp \Aa a_{\pp} + \Box_{\gL} a_{\pp} + \ip{\d \log c}{\d a_{\pp}}_{\gL}) e^{\i \pp \phi},
\end{align}
where $\Hh, \Aa: \Cc^{\infty}(\Mext) \rightarrow \Cc^{\infty}(\Mext)$ are the eikonal operator and the (acoustic) transport operator, respectively, which are defined as
\begin{align*}
    \Hh \phi := \ip{\d \phi}{\d \phi}_{\gL}, \quad
    \Aa a_{\pp} := 2 \ip{\d a_{\pp}}{\d \phi}_{\gL} + (\Box_{\gL} \phi + \ip{\d \log c}{\d \phi}_{\gL}) a_{\pp}.
\end{align*}
We will focus on solving $P v_{\pp} = 0$ in the tubular region $V_{\qq}, \, 0 < \qq \leq 1$ around the null geodesic $\thetaext$ considered where the Fermi coordinates $(s,z')$ are well-defined, and we take as ansatz
\begin{align}
    \phi(\zz,z') &= \sum_{l=0}^N \phi_l(\zz,z'), \quad \phi_l(s,z') = \Phi_{i_1 i_2 \dots i_l}^l(s) z^{i_1} z^{i_2} \dots z^{i_l},
    \label{eq:phi} \\
    a_{\pp}(\zz,z') &= \sum_{m=0}^N \pp^{-m} a_m(\zz,z'), \quad a_m(\zz,z') := \sum_{l=0}^N a_{m,l} (\zz,z'), \quad a_m = A_{i_1 i_2 \dots i_l}^{m,l}(s) z^{i_1} z^{i_2} \cdots z^{i_l}, 
    \label{eq:a_rho}
\end{align}
with smooth, symmetric (for $l \geq 2$) coefficients $\Phi_{i_1 i_2 \cdots i_l}^l(s) \in \CC^{n \times n}$, $A_{i_1 i_2 \cdots i_l}^{m,l}(s) \in \CC^{n \times n}$ for the phase function $\phi$ and the amplitude function $a_{\pp}$, respectively. From equations \eqref{eq:P_WKB}, \eqref{eq:phi}, \eqref{eq:a_rho}, we thus observe that solving $P v_{\pp} = 0$ is equivalent to solving
\begin{align}
    \Hh \phi &= 0, \label{eq:eikonal} \\
    \Tt a_0 &= 0, \label{eq:tr1} \\
    \i \Tt a_m &= - \Box_{\gL} a_{m-1} - \ip{\d (\log c)}{\d a_{m-1}}_{\gL}, \quad m \geq 1, \label{eq:tr2}
\end{align}
in this order, since solving for $a_0$ requires $\phi$, and solving for $a_m, \, m \geq 1$ requires $a_{m-1}$. 
We then define $v_{\pp,\qq}^{(N)}, \, \pp \geq 1, \, 0 < \qq \leq 1$, to be an approximate Gaussian beam solution around $\thetaext$ for $P v = 0$ of order $N=(N_0,N_1,N_2)$ if 
\begin{enumerate}[label=(\roman*)]
    \item equations \eqref{eq:eikonal}, \eqref{eq:tr1}, \eqref{eq:tr2} are solved up to order $N_0,N_1,N_2$ in $z$, respectively,
    \item $\Im(\phi)|_{\thetaext} = 0$,
    \item $\norm{z'}^2 \lesssim \Im(\phi)(s,z')$ for all $(s,z') \in V_{\qq}$,
\end{enumerate}
where the first condition ensures that the linear acoustic equation is approximately solved, where the second condition corresponds to the constant energy of the wave (i.e., no dissipation or absorption of energy), and where the third condition enforces the Gaussian beam solution to concentrate around $\tilde{\t}$. Moreover, we additionally multiply the constructed Gaussian beam with a smooth cut-off function $\x_{\qq}$ for which $\x_{\qq} = 1$ on $\thetaext$ and $\x_{\qq} = 0$ for $\norm{z'} \geq \qq$ to ensure that we stay in $V_{\qq}$, giving that
\begin{align*}
    v_{\pp,\qq}^{(N)} = \x_{\qq} a_{\pp}^{(N_1,N_2)} e^{\i \pp \phi^{(N_0)}}, \quad \phi^{(N_0)} := \sum_{l=0}^{N_0} \phi_l, \quad a_m^{(N_1,N_2)} := \sum_{l=0}^{N_1} a_{m,l}, \, m=0,1,\dots,N_2,
\end{align*}
where we note that the inclusion of a cut-off function $\x_{\qq}$ can be done without loss of generality as the Gaussian beam solutions are concentrated around $\thetaext$ by condition (iii), and as the limit $\qq \to 0$ will be taken later.

The following result shows that the constructed solution indeed approximately solves $P v = 0$, provided $N_k, k=0,1,2$ are large enough. The result is similar to \cite{feizmohammadi2022recovery}[Lemma 2].

\begin{lemma} \label{lem:WKB_convergence}
    Consider a null geodesic $\t$ in $(M,\gL)$. Let $v_{\pp,\qq}^{(N)}, \, \pp \geq 1, \, 0 < \qq \leq 1$, be an approximate Gaussian beam around $\theta$ for $P v = 0$ of order $N=(N_0,N_1,N_2)$ with $N_0 \leq N_1 + 2$, $N_1 \leq 2 N_2 + 1$, and let $r_{\pp,\qq}^{(N)}$ be the corresponding residual, i.e., $r_{\pp,\qq}^{(N)} := v - v_{\pp,\qq}^{(N)}$. Then 
    \begin{align*}
    \norm{r_{\pp,\qq}^{(N)}}_{H^{q+1}(M)} \lesssim \norm{P v_{\pp,\qq}^{(N)}}_{H^q(M)} \lesssim \pp^{\frac{1}{4}(4 q - 2 N_0 + 3)} , \quad \norm{r_{\pp,\qq}^{(N)}}_{L^2(M)} \lesssim \pp^{\frac{1}{4}(- 2 N_0 + 3)},
    \end{align*}
    for any $q \geq 0$.
\end{lemma}

\begin{proof}
Let $\a$ be a 4-dimensional multi-index. Since we solve equations \eqref{eq:eikonal}, \eqref{eq:tr1}, \eqref{eq:tr2} up to order $N_k, \, k=0,1,2$, respectively, we have
\begin{align*}
    |\p^{\a} P v_{\pp,\qq}^{(N)}| \lesssim \pp^{|\a|} |e^{\i \pp \phi^{(N_0)}}| (\pp^2 \norm{z'}^{N_0+1} + \pp \norm{z'}^{N_1+1} + \pp^{-N_2}).
\end{align*}
Let 
\begin{align*}
    \Ii(\norm{z'}^k) := \int_{\norm{z} < \qq} \norm{z'}^k e^{-2C\pp \norm{z'}^2} \d z',
\end{align*}
and note that $\Ii(\norm{z'}^k) \leq \int_0^\infty e^{-2 C \pp r^2} r^{k+2} \d r \lesssim \pp^{-(k+3)/2}, \, k \geq 0$.
Since $|e^{\i \pp \phi^{(N_0)}}| \leq e^{-C\pp \norm{z'}^2}$ by condition (iii), we get
\begin{align*}
    \norm{\p^\a P v_{\pp,\qq}^{(N)}}_{L^2(\M)}^2 &\lesssim \pp^{2|\a|+4} \Ii(\norm{z'}^{2(N_0+1)}) + \pp^{2|\a|+3} \Ii(\norm{z'}^{N_0+N_1+2}) + \pp^{2|\a|+2} \Ii(\norm{z'}^{2(N_1+1)}) \\
    & + \pp^{2|\a|+2-N_2} \Ii(\norm{z'}^{N_0+1}) + \pp^{2|\a|+1-N_2} \Ii(\norm{z'}^{N_1+1}) + \pp^{2|\a|-2 N_2} \Ii(1) \\
    &\lesssim \pp^{2|\a| - N_0 + \frac 3 2},
\end{align*}
provided $N_0 \leq N_1 + 2$, $N_1 \leq 2 N_2 + 1$. We thus get that $\norm{P v_{\pp,\qq}^{(N)}}_{H^{|\a|}(\M)} \lesssim \pp^{\frac{1}{4}(4|\a| - 2 N_0 + 3)}$ for any $|\a| \geq 0$, giving the desired estimate for $\norm{P v_{\pp,\qq}^{(N)}}_{H^q(M)}$ by interpolating Sobolev spaces (see, e.g., \cite{bergh2012interpolation}[Theorem 6.4.5]). Finally, since the remainder $r_{\pp}$ satisfies
\begin{align*}
    \begin{dcases}
    P r_{\pp,\qq}^{(N)} = - P v_{\pp,\qq}^{(N)}, &\text{ in } \M, \\
    r_{\pp,\qq}^{(N)} \equiv 0, &\text{ on } \Sigma, \\
    r_{\pp,\qq}^{(N)} = \p_t r_{\pp,\qq}^{(N)} \equiv 0, &\text{ on } \M_0,
    \end{dcases}
\end{align*}
we get from Corollary \ref{cor:estimate1} that $\norm{r_{\pp,\qq}^{(N)}}_{H^{q+1}(M)} \lesssim \norm{r_{\pp,\qq}^{(N)}}_{E^{q+1}(M)} \lesssim \norm{P v_{\pp,\qq}^{(N)}}_{H^{q}(M)}$ for any $q \geq 0$, and that $\norm{r_{\pp,\qq}^{(N)}}_{L^2(M)} \lesssim \norm{P v_{\pp,\qq}^{(N)}}_{L^2(M)}$.
\end{proof}

We note that an approximate Gaussian beam of order $(2,0,0)$ thus already suffices to solve $Pv = 0$ in $L^2(M)$ as the limit $\pp \to \infty$ will be taken later. Moreover, the following bounds can be derived for $v_{\pp,\qq}^{(N)}$, which we will need later.

\begin{lemma} \label{lem:gb_est} 
Consider a null geodesic $\t$ in $(M,\gL)$. Let $v_{\pp,\qq}^{(N)}, \, \pp \geq 1, \, 0 < \qq \leq 1$, be an approximate Gaussian beam around $\theta$ for $P v = 0$ of order $N$, and let $f_{\pp,\qq}^{(N)} := v_{\pp,\qq}^{(N)}|_{\Sigma}$. Then
\begin{align*}
\norm{f_{\pp,\qq}^{(N)}}_{\Cc^q(\Sigma)} \lesssim \norm{v_{\pp,\qq}^{(N)}}_{\Cc^q(\M)} \lesssim \pp^{q}, 
\quad \norm{v_{\pp,\qq}^{(N)}}_{H^q(\M)} \lesssim \pp^{q - \frac 3 4}, 
\quad \norm{f_{\pp,\qq}^{(N)}}_{H^r(\Sigma)} \lesssim \pp^{r - \frac 1 4},
\end{align*}
for any $q \geq 0, \, r > 0$.
\end{lemma}

\begin{proof}
Let $\a$ be a 4-dimensional multi-index. We have
\begin{align*}
    |\p^{\a} v_{\pp,\qq}^{(N)}| \lesssim \pp^{|\a|} |e^{\i \pp \phi^{(N_0)}}|,
\end{align*}
and we recall that $|e^{\i \pp \phi^{(N_0)}}| \leq e^{-C\pp \norm{z'}^2}$ from condition (iii). Hence, we immediately observe the first estimate, and the estimate for $\norm{v_{\pp,\qq}^{(N)}}_{H^q(\M)}$ follows a similar approach as proving the estimate for $\norm{P v_{\pp,\qq}^{(N)}}_{H^q(M)}$ in Lemma \ref{lem:WKB_convergence} by observing that
\begin{align*}
    \norm{\p^\a v_{\pp,\qq}^{(N)}}_{L^2(\M)}^2 \lesssim 
    \pp^{2|\a|} \Ii(1) \lesssim \pp^{2|\a| - \frac 3 2}.
\end{align*}
which gives $\norm{v_{\pp,\qq}^{(N)}}_{H^{|\a|}(\M)} \lesssim \pp^{|\a| - \frac 3 4}$ for any $|\a| \geq 0$, and thus the desired estimate for $\norm{v_{\pp,\qq}^{(N)}}_{H^q(M)}$ by interpolating Sobolev spaces. Finally, by the trace theorem, $\norm{f_{\pp,\qq}^{(N)}}_{H^{r}(\Sigma)} \lesssim \norm{v_{\pp,\qq}^{(N)}}_{H^{r+1/2}(\M)}$.
\end{proof}

\subsection{Construction phase function}
We will construct the phase function for $v_{\pp,\qq}^{(N)}$ by solving equation \eqref{eq:eikonal} up to order $N_0$. Recall that the ansatz for $\phi$ was taken as in equation \eqref{eq:phi}. We construct the first three terms $\Phi^j, j=0,1,2$ explicitly, as we know that $N_0 \geq 2$ for convergence of the Gaussian beam. For convenience of notation, we denote $\dot{\Phi}^j = \p_0 \Phi^j$ and $K = \Phi^0, D = \Phi^1, H = \Phi^2$.

Grouping terms of order $N_0 = 0,1$ in the eikonal equation $\Hh \phi = 0$ gives
\begin{align*}
    2 \dot{K} D_1 + D_2^2 + D_3^2 &= 0, \\
    \dot{D}_i D_1 + \dot{K} H_{i1} + D_2 H_{i2} + D_3 H_{i3} &= 0,
\end{align*}
for which we note that a solution is given by $K = 0,\, D = (1,0,0)$, i.e.,
\begin{align*}
    \phi_0(s,z') = 0, \quad \phi_1(s,z') = z^1. 
\end{align*}

Going to $N_0 = 2$, the second-order term $G^2(s)$ of the metric also comes into play. Grouping the terms for $N_0 = 2$ and imposing $K = 0,\, D = (1,0,0)$, we get
\begin{align*}
    \dot{H}_{ij} + \sum_{k=2,3} H_{i k} H_{j k} + G_{11}^2 = 0, 
\end{align*}
or, written in matrix-form,
\begin{align*}
    \dot{H} + H \C H + \D = 0, 
    \quad H(\zz_-) = H_0,
    \quad \C_{ij} := 
    \begin{dcases}
        1, \quad i=j=2,3, \\
        0, \quad \text{otherwise},
    \end{dcases}
    \quad
    \D := G_{11}^2,
\end{align*}
which is a Ricatti equation. Recall that we require $\Im(H(s)) > 0$ by property (iii), and by taking $\Im(H_0) > 0$ this property is guaranteed through the above Ricatti equation (see, e.g., \cite{kachalov2001inverse}[Lemma 2.56]). Moreover, by setting $H_0 = Z_0 Y_0^{-1}$, the Ricatti equation can be rewritten as a system of linear equations
\begin{align*}
    \dot{Y}(\zz) &= \C Z(\zz), \quad Y(\zz_-) = Y_0, \\
    \dot{Z}(\zz) &= -\D Y(\zz), \quad Z(\zz_-) = Z_0,
\end{align*}
where $H(s) = Z(s) Y(s)^{-1}$. Regarding initial conditions, we set $Y_0 = I$, $Z_0 = \i I$, implying that $\Im(H_0) > 0$. We note that $Y$ satisfies 
\begin{align} \label{eq:Y}
    \ddot{Y} + \C \D Y = 0, \quad Y(\zz_-) = Y_0, \quad \dot{Y}(\zz_-) = \C Z_0, 
\end{align}
which is related to the Jacobi field equation, see, e.g., \cite{acosta2022nonlinear}. Moreover, we have
\begin{align} \label{eq:identity_HY}
    \det(\Im(H(\zz))) \cdot |\det Y(\zz)|^2 = C_{\thetaext},
\end{align}
for some $C_{\thetaext} > 0$ dependent on $\thetaext$ \cite{kachalov2001inverse}[Lemma 2.58]. From this identity and the fact that $\Im(H(s)) > 0$, we also observe that $\det Y(\zz) \neq 0$. We note that $\det Y(\zz) \neq 0$ along $\thetaext$ does not exclude $\thetaext$ from having conjugate points, as this is caused by the fact that we lifted the construction into the complex domain.

Repeating this procedure, the higher-order terms $N_0 = l \geq 3$ can be observed to obey linear ODEs of the form 
\begin{align*}
    \frac{2}{l!} \dot{\Phi}_{i_1 i_2 \cdots i_l}^l + 2 \sum_{k=2,3} \Phi_{i_1 i_2 \cdots i_{l-1} k}^l H_{i_l k} + (G_{11}^l)_{i_1 i_2 \cdots i_l} = (\Ee^{l})_{i_1 i_2 \cdots i_l}, \quad i_1,\dots,i_l=1,2,3, 
\end{align*}
where $\Ee^{l}$ is an $l$-dimensional tensor whose elements consists of products of terms of $G^{i}, \Phi^{i}, \dot{\Phi}^{i}, \, i < l$. We do not need to solve explicitly for these terms, but we note that these terms are bounded by Picard-Lindel\"of since the metric is smooth. Finally, we note that $\phi$ as constructed above satisfies conditions (ii) and (iii).

\subsection{Construction amplitude function}
We will construct the amplitude function for $v_{\pp,\qq}^{(N)}$ by solving equations \eqref{eq:tr1}, \eqref{eq:tr2} up to order $(N_1,N_2)$.
Recall that the ansatz for $a$ was taken as in equation \eqref{eq:a_rho}.
We first focus on the construction of $a_0$ and construct the term $A^{0,0}$ explicitly. For ease of notation, we set $A^{0,j} = A^j,\, j\geq 0$ throughout this discussion.

Using the earlier derived phase function, we first note that
\begin{align*}
    \Box_{\gLext} \phi = \sum_{i,j=0}^3 \bar{g}^{ij} \p_{ij}^2 \phi &= H_{22} + H_{33} + \left[2 \dot{H}_{i1} + \Phi_{i22}^3 + \Phi_{i33}^3 \right] z^i + \dots
\end{align*}
Furthermore
\begin{align*}
    H_{22} + H_{33} = \Tr(\C H) = \Tr(\dot{Y} Y^{-1}) = \Tr\left( \p_0 \log Y  \right) = \p_0 (\log \det Y).
\end{align*}
Grouping terms of order $N_1 = 0$ in equation \eqref{eq:tr1}, we get
\begin{align*}
    \dot{A}^0 = - \frac{1}{2} [H_{22} + H_{33} + \p_0 (\log c)] = - \frac{1}{2} \p_0(\log(c \det Y)) A^0. 
\end{align*}
The solution to this ODE normalized at $p=(s_-,0) = \thetaext(t_-)$ is then given by
\begin{align} \label{eq:zeta0}
    A^0(s) = c(p)^{\frac{1}{2}} c(\thetaext(s))^{-\frac{1}{2}} (\det Y(s))^{-\frac{1}{2}}.
\end{align}
where we recall that we abuse notation and denote $c(\pi_{\Omegaext}(q)), \, q \in M$ by $c(q)$.

Repeating this procedure, we get that the higher-order terms $N_1 = l \geq 1$ satisfy linear ODEs of the form 
\begin{align*}
    \frac{1}{l!} \dot{A}_{i_1 i_2 \cdots i_l}^l + \sum_{k=2,3} A_{i_1 i_2 \cdots i_{l-1} k}^l A_{i_l k}^2 - \p_0(\log (c \det Y)) A_{i_1 i_2 \cdots i_l}^l &= (\Tt^{l})_{i_1 i_2 \cdots i_l}, \quad i_1,i_2,\dots,i_l = 1,2,3,
\end{align*}
where $\Tt^{l}$ is an $l$-dimensional tensor whose elements consists of products of terms of $G^{i},\Phi^{i},\dot{\Phi}^{i},A^j,\dot{A}^j$, $j < i \leq l$. We do not need to solve explicitly for these terms, but we note that these terms are bounded by Picard-Lindel\"of since the metric is smooth.

Regarding $m \geq 1$, although we do not solve explicitly for $A^{m,l}, \, m \geq 1, l \geq 0$, we note that the ODEs governing these terms simply satisfy the above ODE with an additional term $\i P a_{m-1}$ on the right-hand side, and thus additionally depend on $A^{m-1,j}, \, j \leq l$. Again by Picard-Lindel\"of, we can thus conclude that these terms remain bounded.

\section{Proof of Proposition \ref{thm:stabb1}} \label{sec:thm1}

We take the constructed approximate Gaussian beam of order $N$ as solution for $v$, i.e.,
\begin{align} \label{eq:GBu12}
    v = \x_{\qq} a_{\pp} e^{\i \pp \phi}, \quad f = v|_{\Sigma},
\end{align}
where $\phi = \phi^{(N_0)}$, $a_{\pp} = a^{(N_1,N_2)}$, where we set $N \geq (2,0,0)$ to guarantee that $P v = 0$ in the limit $\pp \to \infty, \, \qq \to 0$, and where we drop the dependence of $v$ on $\pp,\qq,N$ for ease of notation. Regarding $\v$, we take the complex conjugate of this Gaussian beam solution (which is also a Gaussian beam solution as $P v = 0$ implies that $P \bar{v} = 0$) with $\pp$ doubled to establish a phase cancellation later, i.e.,
\begin{align} \label{eq:GBv}
    \v = \x_{\qq} \bar{a}_{2 \pp} e^{-2 \i \pp \bar{\phi}}, \quad \ff = \v|_{\Sigma}.
\end{align}
We also note that $\d t \d x = c^3(1 + \Oo(\qq^2)) \, \d \zz \wedge \d z'$ is the Euclidean volume form under Fermi coordinates, since $|\det \bar{g}(t,x) |^{1/2} = c^3$ and $|\det \bar{g}(s,z')|^{1/2} = 1 + \Oo(\qq^2)$.
As a consequence,
\begin{align} \label{eq:sp1}
    \int_{\Mext} \b \p_t v^2 \p_t \v \, \d x \d t = \int_{\zz_- - \qq}^{\zz_+ + \qq} \int_{\norm{z'} < \qq} \b e^{- 4 \pp \Im(\phi)} \left[4 \pp^2 |\p_t \phi|^2 \amp_0^2 \bar{\amp}_0 + \Oo(\pp) \right] c^3 ( 1+ \Oo(\qq^2)) \, \d \zz \wedge \d z',
\end{align}
where we extended $\b$ smoothly to $\Omegaext$, in light of the main results of this work. From the constructed phase and amplitude functions, and from the fact that $\p_t z^1 = -\frac{1}{\sqrt{2}}$, we observe that
\begin{align*}
    |\p_t \phi|^2 \amp_0^2 \bar{a}_0
    = \frac{1}{2} c(p)^{\frac 3 2} c(\thetaext(s))^{-\frac{3}{2}} (\det Y(\zz))^{-\frac 3 2} + \Oo(\qq)
\end{align*}
in $V_{\qq}$. This combined with Laplace's method to approximate the inner integral in equation \eqref{eq:sp1} (see, e.g., \cite{miller2006applied}[Section 3.7]), equation \eqref{eq:identity_HY}, the fact that $\det Y(\zz) \neq 0$ along $\thetaext$, and conditions (ii) and (iii) gives that
\begin{align} \label{eq:sp2}
\begin{split}
    \int_{\norm{z'} < \qq} \b e^{- 4 \pp \Im(\phi)} [4 \pp^2 |\p_t \phi|^2 \amp_0^2 \bar{\amp}_0 + \Oo(\pp) ]& c^3 ( 1+ \Oo(\qq^2)) \, \d z' \\
    &= C_{c,\thetaext,p} \b(\thetaext(s)) c^{\frac 3 2}(\thetaext(s)) (\det Y(\zz))^{-\frac 1 2} \pp^{\frac 1 2} + \Oo(\pp^{-\frac 1 2}),
\end{split}
\end{align}
for any $\pp \geq 1, \, 0 < \qq \leq 1$, where $C_{c,\thetaext,p} = 2 (2\pi)^{3/2}C_{\thetaext}^{-1/2} c(p)^{\frac 3 2}$.
Combining equations \eqref{eq:sp1}, \eqref{eq:sp2} gives
\begin{align} \label{eq:wv_integral_identity_0}
    \pp^{-1/2} \int_{M} \b  \p_t v^2 \p_t \v \, \d x \d t = C_{c,\thetaext} \int_{\zz_- - \qq}^{\zz_+ + \qq} \b(\thetaext(s)) c^{\frac 3 2}(\thetaext(s)) (\det Y(\zz))^{-\frac 1 2} \, \d \zz + \Oo(\pp^{-1}),
\end{align}
for any $\pp \geq 1, \, 0 < \qq \leq 1$, where we used that $c,\b$ are bounded on $\Omegaext$ and that $\det Y$ does not vanish along $\tilde{\t}$. 
Let $\g, \t$ be the restrictions of $\gammaext, \thetaext$ to $\Omega,M$, respectively, and define
\begin{align*}
    \Jj_{c,\theta} \b := C_{c,\theta} \int_{\zz_-}^{\zz_+} \b(\t(s)) c(\t(s))^{\frac{3}{2}} (\det Y(\zz))^{-\frac 1 2} \, \d \zz
\end{align*}
which is a weighted geodesic ray transform of $\b$, in particular, the Jacobi weighted ray transform of the first kind (see, e.g., \cite{feizmohammadi2020inverse}). We then get from equation \eqref{eq:wv_integral_identity_0} that
\begin{align} \label{eq:wv_integral_identity}
    \pp^{-1/2} \int_{\M} \b \p_t v^2 \p_t \v \, \d x \d t
    = \Jj_{c,\theta} \b + \Oo(\pp^{-1} + \qq)
\end{align}
for any $\pp \geq 1, \, 0 < \qq \leq 1, \, N \geq (2,0,0)$, where we again used that $c,\b$ are bounded on $\Omegaext$ and that $\det Y$ does not vanish along $\tilde{\t}$. From this integral identity, Proposition \ref{thm:stabb1} readily follows.

\begin{proof}[Proof Proposition \ref{thm:stabb1}] 
Recall that $\g,\t$ are uniquely parameterized by $(p = \gamma(t_-), \xi' = \dot{\gamma}(t_-)) \in S_-^* \p \Omega$, giving that $\Jj_{c,\t} = \Jj_{c,p,\xi'}$. We set $v, \v, f, h$ as in equations \eqref{eq:GBu12}, \eqref{eq:GBv} with $N \geq (2,0,0)$. From equation \eqref{eq:wv_integral_identity}, we then have that
\begin{align*}
    |\Jj_{c,p,\xi'}(\b_1 - \b_2)| \leq \pp^{-1/2} \left| \int_M (\b_1 - \b_2) \p_t v^2 \p_t \v \, \d x \d t \right| + \Oo(\pp^{-1} + \qq)
\end{align*}
for any $\pp \geq 1, \, 0 < \qq \leq 1$. Moreover, from equation \eqref{eq:wv_alessandrini} and Lemma \ref{lem:DN12}
\begin{align*}
    \Bigg| \int_\M (\b_1 - \b_2) \p_t v^2 \p_t \v \, \d x \d t \Bigg| &\leq |\ip{(\p_{\eps}^2 \Lambda_1 - \p_{\eps}^2 \Lambda_2)(\eps f)}{\ff}_{L^2(\Sigma)}| \\
    &\lesssim \norm{\ff}_{H^{-(s-2)}(\Sigma)} \norm{(\p_{\eps}^2 \Lambda_1 - \p_{\eps}^2 \Lambda_2)(\eps f)}_{H^{s-2}(M)} \\
    &\lesssim \norm{\ff}_{\Cc(\Sigma)} (\eps^{-2} \delta + \eps \norm{f}_{H^{s+3}(\Sigma)}^3).
\end{align*}
where $\Lambda_{\ind}$ is the DN map corresponding to $\b=\b_\ind$, and we have from Lemma \ref{lem:gb_est} that $\norm{\ff}_{\Cc(\Sigma)} \lesssim 1, \, \norm{f}_{H^{s+3}(\Sigma)} \lesssim \pp^{\frac{4s + 11}{4}}$. Combining the above estimates and additionally setting $\qq = \pp^{-1}$ gives that
\begin{align*}
    |\Jj_{c,p,\xi'}(\b_1 - \b_2)| \lesssim \pp^{-\frac 1 2} \eps^{-2} \delta + \eps \pp^{\frac{12s+31}{4}} + \pp^{-1}
\end{align*}
for any $\pp \geq 1$. We will bound the right-hand side in terms of $\delta$ by finding the minimizers $\pp,\eps$ in terms of $\delta$ that do so. Setting $\pp,\eps$ to be these minimizers gives that
\begin{align*}
    \pp = C_s \delta^{-\frac{1}{6s+18}}, \quad \eps = \tilde{C}_s \delta^{\frac{12s+35}{24s+72}},
\end{align*}
with $C_s,\tilde{C}_s > 0$ some constants dependent on $s$ only. With these choices for $\pp,\eps$, we get from Lemma \ref{lem:gb_est} that $\eps \norm{f_{\pp}}_{H^{s+3}(\Sigma)} \lesssim \delta^{\frac{1}{12s+36}}$, guaranteeing well-posedness of equation \eqref{eq:wv} for $\delta$ small enough, and that
\begin{align*}
    |\Jj_{c,p,\xi'}(\b_1 - \b_2)| \lesssim \delta^{\frac{1}{6s+18}}.
\end{align*}
As the geodesic was chosen arbitrarily, we thus have that
\begin{align*}
    \norm{\Jj_{c,p,\xi'}(\b_1 - \b_2)}_{L^\infty(S_-^*\p\Omega)} \lesssim \delta^{\frac{1}{6s+18}}.
\end{align*}
Since $(\Omega,g)$ satisfies the foliation condition, we can invert $\Jj_{c,p,\xi'}(\b_1 - \b_2)$ through the procedure described in Section \ref{sec:preliminaries}, and therefore we get that
\begin{align*}
    \norm{\b_1 - \b_2}_{H_\mathsf{F}^{-1}(\Omega)} \lesssim \delta^{\frac{1}{6s+18}}
\end{align*}
from estimate \eqref{eq:stab_wgrt}. Since we have $\norm{\b}_{H^{l}(\Omega)} \leq C_\b$, we can improve the stability estimate by interpolating Sobolev spaces. Choosing $\kappa$ such that $0 < \kappa < \frac{l-q}{l+1}$, we get
\begin{align*}
    \norm{\b_1 - \b_2}_{H^q(\Omega)} &\leq \norm{\b_1 - \b_2}_{H^{l - (l+1) \kappa}(\Omega)} \\
    &\lesssim \norm{\b_1 - \b_2}_{H^{-1}(\Omega)}^{\kappa} \norm{\b_1 - \b_2}_{H^{l}(\Omega)}^{1-\kappa} \\
    &\lesssim \norm{\b_1 - \b_2}_{H_\mathsf{F}^{-1}(\Omega)}^{\kappa} \\
    &\lesssim \delta^{\frac{\kappa}{6s+18}},
\end{align*}
for any $q < l$.
\end{proof}

\section{Proof of Theorem \ref{thm:stabb2}} \label{sec:thm2}

\subsection{Stability of geodesic flow, Jacobi fields, and Gaussian beams}
To prove Theorem \ref{thm:stabb2}, we need to establish stability of geodesic flow, of the Jacobi field $Y$, and of the constructed Gaussian beams with respect to the sound speed. The following stability result for variable-coefficient ordinary differential equations will be a key ingredient in proving these results. The proof is a slight extension to Picard-Lindel\"of (see, e.g., \cite{coddington1955theory}[Theorem 7.4]), and is included in Appendix \ref{ap:proof_ODE} for completeness.

\begin{lemma} \label{lem:ODE_est}
Consider the equation
\begin{align*}
    \begin{dcases}
    \dot{x}(t) = f(t,x(t),\a(x(t))), \\ 
    x(t_-) = \xi,
    \end{dcases}
\end{align*}
on a bounded interval $[t_-,t_+]$, and suppose that $x(t) \in \Omega \subset \RR^n$ for all $t \in [t_-,t_+]$, and that $\a:\Omega \rightarrow \RR^m, \, \a \in L^{\infty}(\Omega)$. If $f$ is globally Lipschitz in $x$ and $\a$, then there exists a unique solution $x \in \Cc^1([t_-,t_+])$ with
\begin{align} \label{eq:est_ODE_lipschitz}
    \norm{x_1 - x_2}_{L^{\infty}([t_-,t_+])} \lesssim \norm{\xi_1 - \xi_2}_{} + \norm{ \a_1 - \a_2 }_{L^{\infty}(\Omega)},
\end{align}
where $x_\ind, \, \ind=1,2$ solves the equation with $\xi=\xi_\ind, \, \a=\a_\ind$.
Moreover, if additionally $f$ is $k-1$ times differentiable in $t$, $k$ times differentiable in $x$ and $\a$, and $\a \in \Cc^{k-1}(\Omega)$ with $\norm{\a}_{\Cc^{k-1}(\Omega)} \leq C_\a$ for some $C_\a > 0$,
then $x \in \Cc^{k}([t_-,t_+])$ with
\begin{align} \label{eq:est_ODE_k}
    \norm{x_1 - x_2}_{\Cc^k([t_-,t_+])} \lesssim \norm{\xi_1 - \xi_2}_{} + \norm{ \a_1 - \a_2 }_{\Cc^{k-1}(\Omega)},
\end{align}
for any $k \geq 1$.
\end{lemma}

\begin{remark}
The result of Lemma \ref{lem:ODE_est} can naturally be extended to multiple parameters $\a^j, \, j=1,2,\dots,l$, and the corresponding estimates are given by
\begin{align*}
    \norm{x_1(t) - x_2(t)}_{L^{\infty}([t_-,t_+])} &\lesssim \norm{\xi_1 - \xi_2}_{} + \sum_{j=1}^l \norm{ \a_1^{j} - \a_2^{j} }_{L^\infty(\Omega)}, \\
    \norm{x_1(t) - x_2(t)}_{\Cc^{k}([t_-,t_+])} &\lesssim \norm{\xi_1 - \xi_2}_{} + \sum_{j=1}^l \norm{ \a_1^{j} - \a_2^{j} }_{\Cc^{k-1}(\Omega)}.
\end{align*}
\end{remark}

Lemma \ref{lem:ODE_est} readily proves continuity of the geodesic flow, the Jacobi field $Y$, and the constructed Gaussian beams with respect to the metric, and thus with respect to the sound speed provided it is a priori bounded as in that case
\begin{align} \label{eq:stab_g-c}
    \norm{g_1 - g_2}_{\Cc^k(\Omega)} \lesssim \norm{c_1 - c_2}_{\Cc^k(\Omega)},
\end{align}
for any $k \geq 0$. We will now make this precise.

\begin{lemma} \label{prop:stab_geod_flow_0}
Consider the Riemannian manifolds $(\Omega,g_\ell), \, (g_\ind)_{ij} = c_\ind^{-2} \delta_{ij}, \, \ind = 1,2$, and suppose $c_\ind \in \Cc^{k}(\Omega), \, c_\ind>0, \, k \geq 2$ with $\norm{c_\ind}_{\Cc^k(\Omega)} \leq C_c$ for some $C_c > 0$. Let $\g_\ind(t), \, t \in [t_-,t_+]$ be geodesics in respectively $(\Omega,g_\ind)$ with $\g_1(t_-) = \g_2(t_-)$, $\dot{\g}_1(t_-) = \dot{\g}_2(t_-)$ without reflections with the boundary. Then
\begin{align*}
    \norm{\g_1 - \g_2}_{\Cc^{k+1}([t_-,t_+])} \lesssim \norm{c_1 - c_2}_{\Cc^k(\Omega)}.
\end{align*}
\end{lemma}

\begin{proof}
Consider the geodesic flow equation
\begin{align*}
    \ddot{\g}_\ind^k + \Gamma_{ij}^k(\g) \dot{\g}_\ind^i \dot{\g}_\ind^j = 0,
\end{align*}
on $[t_-,t_+]$. Since the Christoffel symbols $\Gamma_{ij}^l$ are given by $\Gamma_{ij}^l = \frac{1}{2} g^{lm} (\p_i g_{jm} + \p_j g_{im} - \p_m g_{ij})$, we have that $\norm{\Gamma_\ind}_{\Cc^{k-1}(\Omega)} \leq C_c$ and
\begin{align*}
    \norm{ \Gamma_1 - \Gamma_2 }_{\Cc^{k-1}(\Omega)} \lesssim \norm{g_1 - g_2}_{\Cc^{k}(\Omega)} \lesssim \norm{c_1 - c_2}_{\Cc^{k}(\Omega)}
\end{align*}
from estimate \eqref{eq:stab_g-c}. In order to apply Lemma \ref{lem:ODE_est}, we rewrite the geodesic flow equation to a first-order equation $\dot{\z}_\ind = f(\z_\ind;\Gamma_\ind)$ by setting $\z_\ind := (\g_\ind^1,\g_\ind^2,\g_\ind^3,\dot{\g}_\ind^1,\dot{\g}_\ind^2,\dot{\g}_\ind^3)$. Since $f$ is smooth, $\norm{\Gamma_\ind}_{\Cc^{k-1}(\Omega)} \leq C_c$, and $\z_1(t_-) = \z_2(t_-)$, we get from Lemma \ref{lem:ODE_est} that
\begin{align*}
    \norm{\gamma_1 - \gamma_2}_{\Cc^{k+1}([t_-,t_+])} \lesssim \norm{ \Gamma_1 - \Gamma_2 }_{\Cc^{k-1}(\Omega)}.
\end{align*}
Combining the above estimates concludes the proof. 
\end{proof}

\begin{corollary} \label{prop:stab_geod_flow}
    Consider the Lorentzian manifolds $(\M, \gL_\ell), \, (\gL_\ind)_{ij} = -\d t^2 + c_\ind^{-2} \delta_{ij}, \, \ind = 1,2$, and suppose $c_\ind \in \Cc^{k}(\Omega), \, c_\ind>0, \, k \geq 2$ with $\norm{c_\ind}_{\Cc^k(\Omega)} \leq C_c$ for some $C_c > 0$. Let $\t_\ind(s), \, s \in [s_-,s_+]$ be null geodesics in respectively $(\M,\gL_\ind)$ with $\t_1(s_-) = \t_2(s_-)$, $\dot{\t}_1(s_-) = \dot{\t}_2(s_-)$ without reflections with the boundary. Then
    \begin{align*}
        \norm{\t_1 - \t_2}_{\Cc^{k+1}([s_-,s_+])} \lesssim \norm{c_1 - c_2}_{\Cc^k(\Omega)}.
    \end{align*}
\end{corollary}

\begin{lemma} \label{lem:Y}
Consider the Lorentzian manifolds $(\M, \gL_\ell), \, (\gL_\ind)_{ij} = -\d t^2 + c_\ind^{-2} \delta_{ij}, \, \ind = 1,2$, and suppose $c_\ind \in \Cc^{k}(\Omega), \, c_\ind>0, \, k \geq 2$ with $\norm{c_\ind}_{\Cc^k(\Omega)} \leq C_c$ for some $C_c > 0$. Let $Y_{\ind}$ be the Jacobi fields obeying equation \eqref{eq:Y} along null geodesics $\t_\ind(s), \, s \in [s_-,s_+]$ with $\t_1(s_-) = \t_2(s_-)$, $\dot{\t}_1(s_-) = \dot{\t}_2(s_-)$, where $\t_\ind$ do not reflect with the boundary in $[s_-,s_+]$. Then
\begin{align*}
    \norm{Y_1 - Y_2}_{\Cc^k([s_-,s_+])} \lesssim 
    \norm{c_1 - c_2}_{\Cc^{k}(\Omega)}.
\end{align*}
In particular
\begin{align*}
    \norm{\det(Y_1) - \det(Y_2)}_{\Cc^k([s_-,s_+])} \lesssim \norm{c_1 - c_2}_{\Cc^{k}(\Omega)}.
\end{align*}
\end{lemma}

\begin{proof}
Let $\D_\ind$ be the matrix $\D = G_{11}^2$ appearing in equation \eqref{eq:Y} in the metric $g_\ind$, then $\norm{\D_\ind}_{\Cc^{k-2}(\Omega)} \leq C_c$ and
\begin{align} \label{eq:est_Bgc}
    \norm{\D_1 - \D_2}_{\Cc^{k-2}([s_-,s_+])} \lesssim \norm{g_1 - g_2}_{\Cc^{k}(\Omega)} \lesssim \norm{c_1 - c_2}_{\Cc^{k}(\Omega)}
\end{align}
from estimate \eqref{eq:stab_g-c}. Following the same strategy as in the proof of Lemma \ref{prop:stab_geod_flow_0}, i.e., rewriting $\ddot{Y} = f(Y,\D)$ as a first-order ODE, applying Lemma \ref{lem:ODE_est}, and using $Y_{1}(s_-) = Y_{2}(s_-) = Y_0, \, \dot{Y}_{1}(s_-) = \dot{Y}_{2}(s_-) = AZ_0$, we get that
\begin{align*}
    \norm{Y_1 - Y_2}_{\Cc^{k}([s_-,s_+])} 
    \lesssim \norm{\D_1 - \D_2}_{\Cc^{k-2}([s_-,s_+])}.
\end{align*}
Hence, the first estimate follows. Finally, the second estimate readily follows from the definition of the determinant.
\end{proof}

\begin{lemma} \label{lem:stab_gb}
Consider the Lorentzian manifolds $(\M, \gL_\ell), \, (\gL_\ind)_{ij} = -\d t^2 + c_\ind^{-2} \delta_{ij}, \, \ind = 1,2$, and suppose $c_\ind \in \Cc^\infty(\Omega), \, c_\ind> 0$ with $\norm{c_\ind}_{\Cc^{k+1}(\Omega)} \leq C_c$ for some $k \geq 1, \, C_c > 0$.
Let $v_\ind = (v_{\pp,\qq}^{(N)})_{\ind}, \, \pp \geq 1, \, 0 < \qq \leq 1$ be approximate Gaussian beams of order $N = (2,0,0)$ around the null geodesics $\t_\ind(s), \, s \in [s_-,s_+]$ with $\t_1(s_-) = \t_2(s_-)$, $\dot{\t}_1(s_-) = \dot{\t}_2(s_-)$, where $\t_\ind$ do not reflect with the boundary in $[s_-,s_+]$.
Then
\begin{align*}
\norm{v_1 - v_2}_{\Cc^{k}(M)} \lesssim \pp^{k+1} \norm{c_1 - c_2}_{\Cc^{k+1}(\Omega)}.
\end{align*}
\end{lemma}

\begin{proof}
Let $(s_\ind,z_\ind')$ be the Fermi coordinates corresponding to $\t_\ind$, and let $\g_{\ind}$ be the unit-speed geodesics in $(\Omega, c_{\ind}^{-2} \delta_{ij})$ corresponding to $\t_{\ind}$, where we set $t=t_{\pm}$ such that $\t_1(s_\pm) = \t_1(t_\pm) = (t_\pm,\g_1(t_\pm))$. From Corollary \ref{prop:stab_geod_flow}, we can set $\delta, \qq > 0$ small enough such that the tubular region $V_{\qq}$ around $\t_1$ includes $\t_2$ for $s \in [s_-,s_+]$. We observe that 
\begin{align*}
    |s_1 - s_2| &\leq \frac{1}{\sqrt{2}} \Bigg| \int_{\g_1(t_-)}^{\g_1(t)} c_1(r) \d r - \int_{\g_2(t_-)}^{\g_2(t)} c_2(r) \d r \Bigg| \\
    &\leq \tfrac{1}{\sqrt{2}} ( \text{diam}_{g_1}(\Omega) \norm{c_1 - c_2}_{L^{\infty}(V_{\qq})} + T C_c \norm{\g_1 - \g_2}_{L^{\infty}([t_-,t_+])}) \\
    &\lesssim \norm{c_1 - c_2}_{L^\infty(\Omega)}
\end{align*}
where we used Lemma \ref{prop:stab_geod_flow_0}, and similarly
\begin{align*}
    |z_1^i - z_2^i| \lesssim \norm{c_1 - c_2}_{L^\infty(\Omega)},
\end{align*}
showing continuity of Fermi coordinates with respect to $c$.

First, we will prove stability of the phase function $\phi^{(2)}$ with respect to $c$. Recall that $\phi_{\ind}^{(2)}(s_\ind,z_\ind') = s_\ind + \frac{1}{2} (H_\ind (s_\ind))_{ij} z_\ind^i z_\ind^j$. Since $\dot{H} = f(H,\D)$ and $\norm{\D}_{\Cc^{k-2}(\M)} \lesssim C_c$, we thus get from Lemma \ref{lem:ODE_est} and estimate \eqref{eq:est_Bgc} that $\norm{H}_{\Cc^{k-1}(\M)} \lesssim C_c$ and
\begin{align*}
    \norm{H_1 - H_2}_{\Cc^{k}([s_-,s_+])} \lesssim \norm{\D_1 - \D_2}_{\Cc^{k-1}([s_-,s_+])} \lesssim \norm{c_1 - c_2}_{\Cc^{k+1}(\Omega)}.
\end{align*}
We have
\begin{align*}
    \norm{\phi_1^{(2)} - \phi_2^{(2)}}_{L^{\infty}(V_{\rho})} &\lesssim |s_1 - s_2| + \qq \norm{z_1' - z_2'} + \qq^2 \norm{H_1 - H_2}_{L^{\infty}([s_-,s_+])},
\end{align*}
and for any 4-dimensional multi-index $\a$ with respectively $|\a| = 1$, $|\a| = l \geq 2$ we have
\begin{align*}
    \norm{\p_z^\a(\phi_1^{(2)} - \phi_2^{(2)})}_{L^{\infty}(V_{\rho})} &\lesssim \norm{z_1 - z_2} + \qq \norm{H_1 - H_2}_{L^\infty([s_-,s_+])} + \qq^2 \norm{H_1 - H_2}_{\Cc^{1}([s_-,s_+])} \\
    \norm{\p_z^\a(\phi_1^{(2)} - \phi_2^{(2)})}_{L^{\infty}(V_{\rho})} &\lesssim \norm{z_1 - z_2} + \norm{H_1 - H_2}_{\Cc^{l-2}([s_-,s_+]} + \qq \norm{H_1 - H_2}_{\Cc^{l-1}([s_-,s_+])} \\
    & \qquad + \qq^2 \norm{H_1 - H_2}_{\Cc^{l}([s_-,s_+])},
\end{align*}
thus giving that
\begin{align*}
    \norm{\phi_1^{(2)} - \phi_2^{(2)}}_{\Cc^k(V_{\qq})} \lesssim \norm{c_1 - c_2}_{\Cc^{k+1}(\Omega)}
\end{align*}
for $k \geq 1$.

Recall that $a_\ind^{(0,0)}(s_\ind,z_\ind) = A^{0,0}(s_\ind) = c(s_-,0)^{\frac{1}{2}} c(s_\ind,0)^{-\frac{1}{2}} (\det Y(s_\ind))^{-\frac{1}{2}}$. Hence, from Lemma \ref{lem:Y} and $\norm{c_\ind}_{\Cc^{k+1}(\Omega)} \leq C_c$, we get that
\begin{align*}
    \norm{a_1^{(0,0)} - a_2^{(0,0)}}_{\Cc^{k}(V_{\qq})} 
    \lesssim \norm{c_1 - c_2}_{\Cc^{k}(\Omega)} + \norm{\det(Y_1) - \det(Y_2)}_{\Cc^k([s_-,s_+])} 
    \lesssim \norm{c_1 - c_2}_{\Cc^{k}(\Omega)}.
\end{align*}

Using 
\begin{align*}
    e^{\i \pp \phi_1^{(2)}} - e^{\i \pp \phi_2^{(2)}} = \i \pp \int_0^1 (\phi_1^{(2)} - \phi_2^{(2)}) e^{\i \pp \left(s \phi_1^{(2)} + (1-s) \phi_2^{(2)}\right)} \d s,  
\end{align*}
we have that
\begin{align*}
    \norm{e^{\i \pp \phi_1^{(2)}} - e^{\i \pp \phi_2^{(2)}}}_{\Cc^k(V_{\qq})} \lesssim \pp^{k+1} \norm{\phi_1^{(2)} - \phi_2^{(2)}}_{\Cc^k(V_{\qq})}.
\end{align*}
Combining this with the earlier derived estimates, we thus get that
\begin{align*}
    \norm{v_1 - v_2}_{\Cc^{k}(\M)} &= \norm{a_1^{(0,0)} e^{\i \pp \phi_1^{(2)}} - a_2^{(0,0)} e^{\i \pp \phi_2^{(2)}}}_{\Cc^{k}(V_{\rho})} \\ 
    &\lesssim \norm{e^{\i \pp \phi_1^{(2)}} - e^{\i \pp \phi_2^{(2)}}}_{\Cc^{k}(V_\rho)} + \norm{a_1^{(0,0)} - a_2^{(0,0)}}_{\Cc^{k}(V_{\rho})} \\
    &\lesssim \pp^{k+1} \norm{\phi_1^{(2)} - \phi_2^{(2)}}_{\Cc^{k}(V_{\rho})} + \pp^{k} \norm{a_1^{(0,0)} - a_2^{(0,0)}}_{\Cc^{k}(V_{\rho})} \\
    &\lesssim \pp^{k+1} \norm{c_1 - c_2}_{\Cc^{k+1}(\Omega)},
\end{align*}
where we used that $\norm{e^{\i \pp \phi_1^{(2)}}}_{\Cc^{k}(V_\rho)} \lesssim \pp^{k}$ by Lemma \ref{lem:gb_est}.
\end{proof}

\subsection{Proof Theorem \ref{thm:stabb2}}

Throughout the proof of Theorem \ref{thm:stabb2}, let $(g_k)_{ij} = c_{k}^{-2} \delta_{ij}, \, k=0,1,2$, and let $\g_0$ be a geodesic in $(\Omega,g_0)$ starting at some point $(p,\xi') \in S_-^*\p \Omega$ at time $t = t_-$. We set $t_- = 0$, and we define $t_+$ be such that $(\g_0(t_+),\dot{\g}_0(t_+)) \in S_+^*\p \Omega$, where $t_+ < \infty$ as $(\Omega,g_0)$ is non-trapping under the foliation condition. Let $\g_\ind, \, \ind=1,2$ be geodesics in the metric $g_\ind$ with $\g_\ind(t_-) = p, \, \dot{\g}_\ind(t_-) = \xi', \, \ind=1,2$. Since we assumed $g_\ind$ to be $\kappa$-close to $g_0$, we know from Lemma \ref{prop:stab_geod_flow_0} and estimate \eqref{eq:stab_g-c} that $\norm{\g_\ind - \g_0}_{\Cc^1(\Omega)} \lesssim \kappa$, thus giving that $\text{diam}_{g_\ind}(\Omega) \leq \text{diam}_{g_0}(\Omega)+\kappa$ for $\kappa$ small enough. As $T > \text{diam}_{g_0}(\Omega)$, we thus know that $\g_\ind$ will intersect with $\p \Omega$ again in $[0,T]$ provided $\kappa$ is small enough, and we denote the corresponding times by $t=t_{\ind,+}$. We also define $\t_k, \, k=0,1,2$ to be the null geodesics in $(M,\bar{g}_\ind), \, \bar{g}_\ind = -\d t^2 + g_\ind$ corresponding to $\g_k$, i.e, $\t_k(t) = (t,\g_k(t))$. We let $(s,z_k')$ be the Fermi coordinates for these null geodesics, and we set $\t_{0}(s_{\pm}) = \t_{0}(t_{\pm})$ and $\t_{\ind}(s_{\ind,\pm}) = \t_{\ind}(t_{\ind,\pm})$ with $s_- = s_{1,-} = s_{2,-}$. Moreover, we recall that $\g_k, \t_k$ are uniquely parameterized by $(p,\xi') \in S_-^* \p \Omega$.

\begin{proof}[Proof Theorem \ref{thm:stabb2}]

We first will prove stable recovery of $c$ from the DN map $\Lambda$ in Step 1. In Step 2, we use this to additionally recover $\b$ stably from the DN map.

\textit{Step 1.} We will denote the first-order DN map for $c=c_\ind$ by $\tilde{\Lambda}_{\ind}$ to avoid clutter with notation, i.e., $\tilde{\Lambda}_{\ind} f = \p_{\nu} v_{\ind}|_{\Sigma}$ with $v_{\ind}$ solving equation \eqref{eq:lin1} with $c=c_{\ind}$.
Let $\Ll_k$ be the scattering relation corresponding to metric $g_k$. We first prove that the scattering relation can be recovered locally from the DN map $\Lambda$, i.e., on a conic neighborhood $U_p$ of some $(p,\xi') \in S_-^* \p \Omega$. Through a similar reasoning as in \cite{stefanov2018inverse}[Theorem 4.3] we then have that 
\begin{align*}
    |(\Ll_1 - \Ll_2)(q,\eta')| \lesssim \delta^{\mu_0},
\end{align*}
for any $(q,\eta') \in U_p$ and for some $0 < \mu_0 < 1$, if
\begin{align} \label{eq:conditions_SY18}
\begin{split}
    \norm{\tilde{\Lambda}_1^* \tilde{\Lambda}_1 - \tilde{\Lambda}_2^* \tilde{\Lambda}_2}_{H^{s+3}(U_p) \rightarrow H^{s-2}(U_p)} &\leq \delta^{\mu_1}, \\
    \norm{\tilde{\Lambda}_1^* y^\mu \tilde{\Lambda}_1 - \tilde{\Lambda}_2^* y^\mu \tilde{\Lambda}_2}_{H^{s+3}(U_p) \rightarrow H^{s-2}(U_p)} &\leq \delta^{\mu_1}, \\
    \norm{\tilde{\Lambda}_1^* \p_{\mu} \tilde{\Lambda}_1 - \tilde{\Lambda}_2^* \p_{\mu} \tilde{\Lambda}_2}_{H^{s+3}(U_p) \rightarrow H^{s-3}(U_p)} &\leq \delta^{\mu_1},
\end{split}
\end{align}
for some $0 < \mu_1 \leq 1$, where $\tilde{\Lambda}_{\ind}^*$ is the $L^2(\Sigma)$-adjoint of $\tilde{\Lambda}_{\ind}$, where $y^\mu$ are the boundary local coordinates near $p$, and where $\p_{\mu}$ denotes $\p_{y^\mu}$. Taking $f_{\eps} := \eps f \in H_{0,\eps_0}^{s+3}(\Sigma)$ to additionally be compactly supported in $U_p$, we get for any such $f$ with $\norm{f}_{H^{s+3}(\Sigma)} = 1$ and any multi-index $\a$ with $|\a| \leq s-2$ that
\begin{align*}
    |\ip{\p_y^\a (\tilde{\Lambda}_1^* y^i \tilde{\Lambda}_1 - \tilde{\Lambda}_2^* y^i \tilde{\Lambda}_2)f_{\eps}}{f_{\eps}}_{L^2(U_p)}|
    &= |\ip{\p_y^\a(\tilde{\Lambda}_1^* y^i \tilde{\Lambda}_1 - \tilde{\Lambda}_2^* y^i \tilde{\Lambda}_2) f_{\eps}}{f_{\eps}}_{L^2(\Sigma)}| \\
    &\leq |\ip{\p_y^\a (y^i (\tilde{\Lambda}_1 - \tilde{\Lambda}_2)(f_{\eps}))}{\p_y^\a (\tilde{\Lambda}_1 f_{\eps})}_{L^2(\Sigma)}| \\
    &\qquad + |\ip{\p_y^\a (y^i \tilde{\Lambda}_2 f_{\eps})}{\p_y^\a((\tilde{\Lambda}_1 - \tilde{\Lambda}_2)(f_{\eps}))}_{L^2(\Sigma)}| \\
    &\leq (\norm{\p_{\nu} v_1|_{\Sigma}}_{H^{s-2}(\Sigma)} + \norm{\p_{\nu} v_2|_{\Sigma}}_{H^{s-2}(\Sigma)}) \norm{(\tilde{\Lambda}_1 - \tilde{\Lambda}_2)(f_{\eps})}_{H^{s-2}(\Sigma)} \\
    &\lesssim \norm{f}_{H^{s-1}(\Sigma)} \left(\eps^{-1} \delta + \eps \norm{f}_{H^{s+3}(\Sigma)}^2 \right) \\
    &\lesssim \eps^{-1} \delta + \eps,
\end{align*}
and similarly 
\begin{align*}
    |\ip{\p_y^\a((\Lambda_1^* \Lambda_1 - \Lambda_2^* \Lambda_2)f_{\eps})}{f_{\eps}}_{L^2(U_p)}| \lesssim \eps^{-1} \delta + \eps,
\end{align*}
where we used Corollary \ref{cor:estimate1} and Lemma \ref{lem:DN12}.
From the same bounding process, we also get that
\begin{align*}
    |\ip{\p_y^\a (\tilde{\Lambda}_1^* \p_{\mu} \tilde{\Lambda}_1 - \tilde{\Lambda}_2^* \p_{\mu} \tilde{\Lambda}_2)f_{\eps}}{f_{\eps}}_{L^2(U_p)}|
    &= |\ip{\p_y^\a(\tilde{\Lambda}_1^* \p_{\mu} \tilde{\Lambda}_1 - \tilde{\Lambda}_2^* \p_{\mu} \tilde{\Lambda}_2) f_{\eps}}{f_{\eps}}_{L^2(\Sigma)}| \\
    &\leq |\ip{\p_y^\a \p_{\mu} ((\tilde{\Lambda}_1 - \tilde{\Lambda}_2)(f_{\eps}))}{\p_y^\a (\tilde{\Lambda}_1 f_{\eps})}_{L^2(\Sigma)}| \\
    &\qquad + |\ip{\p_y^\a \p_{\mu} (\tilde{\Lambda}_2 f_{\eps})}{\p_y^\a((\tilde{\Lambda}_1 - \tilde{\Lambda}_2)(f_{\eps}))}_{L^2(\Sigma)}| \\
    &\leq (\norm{\p_{\nu} v_1|_{\Sigma}}_{H^{s-3}(\Sigma)} + \norm{\p_{\nu} v_2|_{\Sigma}}_{H^{s-2}(\Sigma)}) \norm{(\tilde{\Lambda}_1 - \tilde{\Lambda}_2)(f_{\eps})}_{H^{s-2}(\Sigma)} \\
    &\lesssim \norm{f}_{H^{s-1}(\Sigma)} \left(\eps^{-1} \delta + \eps \norm{f}_{H^{s+3}(\Sigma)}^2 \right) \\
    &\lesssim \eps^{-1} \delta + \eps,
\end{align*}
for any $f_{\eps} \in H_{0,\eps_0}^{s+3}(\Sigma)$ compactly supported in $U_p$ with $\norm{f}_{H^{s+3}(\Sigma)} = 1$ and any multi-index $\a$ with $|\a| \leq s-3$. Setting $\eps = \delta^{\frac{1}{2}}$, which ensures that $f_{\eps} \in H_{0,\eps_0}^{s+3}(\Sigma)$ by choosing $\delta$ sufficiently small, we thus get that conditions \eqref{eq:conditions_SY18} hold with $\mu_1 = \frac{1}{2}$, and thus that
\begin{align*}
    |(\Ll_1 - \Ll_2)(q,\eta')| \lesssim \delta^{\mu_0},
\end{align*}
for any $(q,\eta') \in U_p$. Since $\p \Omega$ is compact, the result can be extended to a global result, i.e.,
\begin{align*}
    \norm{\text{dist}(\Ll_1,\Ll_2)}_{\Cc(S_-^* \p \Omega)} \lesssim \delta^{\mu_0}
\end{align*}
for some $0 < \mu_0 < 1$, where $\norm{\text{dist}(\Ll_1,\Ll_2)}_{\Cc(S_-^* \p \Omega)}$ denotes maximum over all local coordinate charts of $\sup_{(q,\eta') \in U_p} |(\Ll_1 - \Ll_2)(q,\eta')|$.

From the knowledge of the global scattering relation, $c$ can then be stably recovered in $\Omega$ under the foliation condition as proven in \cite{stefanov2016boundary}[Theorem 5.2], giving that
\begin{align*}
    \norm{c_1 - c_2}_{\Cc^{2}(\Omega)} \lesssim \norm{\Ll_1 - \Ll_2}_{\Cc(D)}^{\mu_2}
\end{align*}
for some $0 < \mu_2 < 1$ dependent on $k$, with $D \subset S_-^* \p \Omega$ corresponding to the geodesics in $\Gg$ considered under the foliation condition as discussed in Section \ref{sec:preliminaries}. Hence, we get 
\begin{align} \label{eq:c_delta}
    \norm{c_1 - c_2}_{\Cc^{2}(\Omega)} \lesssim \delta^{\mu}
\end{align}
for some $0 < \mu < 1$ dependent on $k$.

\textit{Step 2.}
Let $\Jj_{k}$ denote the Jacobi weighted ray transform in the metric $g_k$ along $\theta_k$, i.e., $\Jj_{k} \b = \Jj_{c_k,p,\xi'} \b$. Clearly
\begin{align*}
    \Jj_\ind \b_\ind - \Jj_0 \b_0 = (\Jj_\ind - \Jj_0) \b_\ind + \Jj_0 (\b_\ind - \b_0),
\end{align*}
where $\b=\b_0$ corresponds to $c=c_0$. We will first show that $|\Jj_0(\b_\ind - \b_0)| \lesssim \delta^{\mu}$ for some $0 < \mu < 1$ by proving that $|\Jj_\ind \b_\ind - \Jj_0 \b_0| \lesssim \delta^{\mu_1}$ and $|(\Jj_\ind - \Jj_0) \b_\ind| \lesssim \delta^{\mu_2}$ for some $0 < \mu_1, \mu_2 < 1$.

We start by proving that $|\Jj_\ind \b_\ind - \Jj_0 \b_0| \lesssim \delta^{\mu_1}$. From equation \eqref{eq:wv_integral_identity}, we get that
\begin{align*}
    |\Jj_\ind \b_\ind - \Jj_0 \b_0| \leq \pp^{-1/2} \Bigg| \int_{\M} \b_\ind \p_t v_\ind^2 \p_t \v_\ind - \b_0 \p_t v_0^2 \p_t \v_0 \, \d x \d t \Bigg| + \Oo(\pp^{-1} + \qq),
\end{align*}
where $v_k, \v_k$ are the solutions to respectively equations \eqref{eq:lin1}, \eqref{eq:lin0} with $c=c_k$ for some $f=f_k, h=h_k$. Taking the difference of equation \eqref{eq:wv_alessandrini} for $c=c_1,c_2$ gives
\begin{align*}
    \int_{\M} \b_\ind \p_t v_\ind^2 \p_t \v_\ind - \b_0 \p_t v_0^2 \p_t \v_0 \, \d x \d t = I_1 + I_2 + I_3,
\end{align*}
where
\begin{align*}
    I_1 &:= \langle (\p_{\eps}^2 \Lambda_\ind - \p_{\eps}^2 \Lambda_0) (\eps f_\ind), \ff_\ind \rangle_{L^2(\Sigma)}, \\
    I_2 &:= \langle \p_{\eps}^2 \Lambda_0 (\eps f_\ind) - \p_{\eps}^2 \Lambda_0 (\eps f_0), \ff_\ind \rangle_{L^2(\Sigma)}, \\
    I_3 &:= \langle \p_{\eps}^2 \Lambda_0 (\eps f_0), (\ff_\ind - \ff_0) \rangle_{L^2(\Sigma)}.
\end{align*}
We set $v_k, \v_k, f_k, h_k$ as in equations \eqref{eq:GBu12}, \eqref{eq:GBv} with $c=c_k$ and with $N = (2,0,0)$, where we recall that $N=(2,0,0)$ guarantees convergence of the approximate Gaussian beams as $\pp \to \infty$, $\qq \to 0$.
Moreover, we denote $u_{k,i}, \, i=0,1,2$ as the solution to equation \eqref{eq:wv} with boundary profile $\eps f_i$ for $c=c_k, \, \b = \b_k$, where we recall from Lemma \ref{prop:hol} that $u_{k,i} = \eps v_{k,i} + \eps^2 w_{k,i} + \Rr_{k,i}$.

Regarding the first term, we get from Lemma \ref{lem:DN12} and Lemma \ref{lem:gb_est} that
\begin{align*}
    |I_1| &\leq \norm{\ff_\ind}_{H^{-(s-2)}(\Sigma)} \norm{(\p_{\eps}^2 \Lambda_\ind - \p_{\eps}^2 \Lambda_0) (\eps f_\ind)}_{H^{s-2}(\Sigma)} \\
    &\lesssim \norm{\ff_\ind}_{\Cc(\Sigma)} \left( \eps^{-2} \delta + \eps \norm{f_\ind}_{H^{s+3}(\Sigma)}^3 \right) \\
    &\lesssim \eps^{-2} \delta + \eps \pp^{\frac{12s+33}{4}}.
\end{align*}
Regarding the second term, we first note that
\begin{align*}
    \norm{\p_{\nu} (w_{0,\ind} - w_{0,0})|_{\Sigma}}_{L^2(\Sigma)} 
    &\lesssim \norm{\p_{\nu}(\b \p_t^2 (v_{0,\ind}^2 - v_{0,0}^2))|_{\Sigma}}_{L^1([0,T];H^{-1}(\Omega))} \\
    &\lesssim \norm{\b(v_{0,\ind}^2 - v_{0,0}^2)}_{\Cc^1(M)} \\
    &\lesssim \norm{\b}_{\Cc^1(\Omega)} \norm{v_{0,\ind} + v_{0,0}}_{\Cc^1(M)} \norm{v_{0,\ind} - v_{0,0}}_{\Cc^1(M)} \\
    &\lesssim \pp^2 \norm{v_{0,\ind} + v_{0,0}}_{\Cc^1(M)} \norm{c_\ind - c_0}_{\Cc^2(\Omega)}
\end{align*}
from Corollary \ref{cor:estimate1} and Lemma \ref{lem:gb_est}, and therefore we obtain from estimate \eqref{eq:c_delta} that
\begin{align*}
    |I_2| &\leq \norm{h_\ind}_{L^{2}(\Sigma)} \norm{\p_{\eps}^2 \Lambda_0 (\eps f_\ind) - \p_{\eps}^2 \Lambda_0 (\eps f_0)}_{L^{2}(\Sigma)} \\
    &\leq \norm{h_\ind}_{\Cc(M)} \norm{\p_{\nu} (w_{0,\ind} - w_{0,0})|_{\Sigma}}_{L^2(\Sigma)}  \\
    &\lesssim \pp^3 \delta^{\mu}. 
\end{align*}
Regarding the third term, we get from Lemma \ref{prop:hol}, \ref{lem:stab_gb}, and estimate \eqref{eq:c_delta} that
\begin{align*}
    |I_3| &\leq \norm{h_\ind - h_0}_{H^{-(s-2)}(\Sigma)} \norm{\p_{\eps}^2 \Lambda_0 (\eps f_0)}_{H^{s-2}(\Sigma)} \\
    &\lesssim \norm{\eta_\ind - \eta_0}_{\Cc^1(M)} \norm{f_0}_{H^{s+3}(\Sigma)}^2  \\
    &\lesssim \pp^{\frac{4s+15}{2}} \delta^{\mu}.
\end{align*}
From the derived bounds, and additionally setting $\qq = \pp^{-1}$, we get that
\begin{align}
    |\Jj_\ind \b_\ind - \Jj_0 \b_0| 
    \lesssim \pp^{-\frac{1}{2}} \eps^{-2} \delta + \pp^{2s+7} \delta^{\mu} + \eps \pp^{\frac{12s+31}{4}} + \pp^{-1}. \label{eq:to_bound}
\end{align}
As in the proof of Proposition \ref{thm:stabb1}, we need to bound the right-hand side and $\eps \norm{f_k}_{H^{s+3}(\Sigma)}$ in terms of $\delta$. We use the same method to do so as in the proof of Proposition \ref{thm:stabb1}, but since we do not aim for optimal bounds, and since we note that the right-hand side of equation \eqref{eq:to_bound} is bounded by
\begin{align*}
    \lambda(\pp,\eps) = \pp^{2s+7} \eps^{-2} \delta^{\mu} + \eps \pp^{\frac{12s+31}{4}} + \pp^{-1},
\end{align*}
for $\delta,\eps \leq 1, \, \pp \geq 1$, we instead bound this function in terms of $\delta$, where we remark that we can assume $\delta \leq 1$ without loss of generality. The minimizers of $\lambda$ are given by
\begin{align*}
    \pp = C_s \delta^{-\mu\left(\frac{2}{16s+51}\right)}, \quad \eps = \tilde{C}_s \delta^{\mu \left(\frac{16s+45}{48s+153}\right)},
\end{align*}
with $C_s,\tilde{C}_s > 0$ some constants dependent on $s$ only. 
From these minimizers and Lemma \ref{lem:gb_est}, we get that $\eps \norm{f_k}_{H^{s+3}(\Sigma)} \lesssim \delta^{\mu \left(\frac{20 s + 57}{96s+306}\right)}$, and that each term in the right-hand side of equation \eqref{eq:to_bound} is bounded by $\delta^{\mu\left(\frac{2}{16s+51}\right)}$, i.e.,
\begin{align*}
    |\Jj_\ind \b_\ind - \Jj_0 \b_0| \lesssim \delta^{\mu\left(\frac{2}{16s+51}\right)}.
\end{align*}

To prove the estimate $|(\Jj_\ind - \Jj_0) \b_\ind| \lesssim \delta^{\mu_2}$, we first consider the case $s_{0,+} \geq s_{\ind,+}$ and let $(\Jj_\ind - \Jj_0) \b_\ind = \tilde{I}_1 - \tilde{I}_2$, where 
\begin{align*}
    \tilde{I}_1 &:= \int_{s_-}^{s_{\ind,+}} \b_\ind(\t_\ind(s)) c_{\ind}^{\frac{3}{2}}(\t_\ind(s)) (\det Y_\ind(\zz))^{-\frac{1}{2}} - \b_\ind(\t_0(s)) c_0^{\frac{3}{2}}(\t_0(s)) (\det Y_0(\zz))^{-\frac{1}{2}} \, \d \zz, \\
    \tilde{I}_2 &:= \int_{s_{\ind,+}}^{s_{0,+}} \b_\ind(\t_0(s)) c_0^{\frac{3}{2}}(\t_0(s)) (\det Y_0(\zz))^{-\frac{1}{2}} \, \d \zz.
\end{align*}
Since $\b_{\ind} \in \Cc^{\infty}(\Omega), \, \norm{\b_{\ind}}_{\Cc^l(\Omega)} \leq C_\b$, we get from Corollary \ref{prop:stab_geod_flow} that 
\begin{align*}
    \norm{\b_\ind(\t_\ind(s)) - \b_\ind(\t_0(s))}_{\Cc([s_-,s_{1,+}])} \lesssim \norm{\t_\ind - \t_0}_{\Cc([s_-,s_{1,+}])} \lesssim \norm{c_\ind - c_0}_{\Cc^{2}(\Omega)}.
\end{align*}
This estimate together with estimate \eqref{eq:c_delta}, Lemma \ref{lem:Y}, and the a priori bound on the sound speed gives that
\begin{align*}
    |\tilde{I}_1|
    &\lesssim \norm{\b_\ind(\t_\ind(s)) - \b_\ind(\t_0(s))}_{\Cc([s_-,s_{1,+}])} + \norm{c_\ind(\t_\ind(s)) - c_0(\t_0(s))}_{\Cc([s_-,s_{1,+}])} \\
    &\qquad + \norm{\det Y_\ind(s) - \det Y_0(s)}_{\Cc([s_-,s_{1,+}])} \\
    &\lesssim \norm{c_\ind - c_0}_{\Cc^2(\Omega)} \\
    &\lesssim \delta^{\mu}.
\end{align*}
Again using the a priori boundedness of $\b_\ind, c_0, \det Y_0$, we get from the continuity of Fermi coordinates with respect to $c$ as shown in the proof of Lemma \ref{lem:stab_gb} combined with estimate \eqref{eq:c_delta} that
\begin{align*}
    |\tilde{I}_2| \lesssim |s_{\ind,+} - s_{+}| \lesssim \delta^{\mu}.
\end{align*}
A similar bounding process for $\tilde{I}_1,\tilde{I}_2$ applies in case $s_{0,+} \leq s_{\ind,+}$. Combining the estimates for $\tilde{I}_1, \tilde{I}_2$ gives that $|(\Jj_\ind-\Jj_0)\b_\ind| \lesssim \delta^{\mu}$.

From the above, we thus get that $|\Jj_0(\b_\ind - \b_0)| \lesssim \delta^{\tilde{\mu}}$ for some $0 < \tilde{\mu} < 1$ dependent on $k,s$. Since $(p,\xi') \in S_-^* \p \Omega$ was arbitrary, we thus obtain that
\begin{align*}
    \norm{\Jj_0(\b_\ind - \b_0)}_{\Cc(S_-^* \p \Omega)} \lesssim \delta^{\tilde{\mu}}.
\end{align*} 
Since $(\Omega,g_0)$ satisfies the foliation condition, we can invert $\Jj_0(\b_\ind - \b_0)$ through the procedure described in Section \ref{sec:preliminaries}, with estimate \eqref{eq:stab_wgrt} giving that $\norm{\b_\ind - \b_0}_{H_{\mathsf{F}}^{-1}(\Omega)} \lesssim \delta^{\tilde{\mu}}$. Hence, we have
\begin{align*}
    \norm{\b_1 - \b_2}_{H_{\mathsf{F}}^{-1}(\Omega)} \lesssim \delta^{\tilde{\mu}}.
\end{align*}
The same Sobolev interpolation argument as in the proof of Proposition \ref{thm:stabb1} concludes the proof.
\end{proof}

\begin{remark}
We note that we recover $c$ in Step 1 by first recovering the scattering relation from the (first-order) DN map rather than from the DN map directly. This is due to the fact that the scattering relation depends continuously on $c$, whereas the DN map itself might not, see, e.g., \cite{bao2014sensitivity}.
\end{remark}

\section{Numerical experiments}\label{sec:numerical_experiments}

We will numerically investigate the stable recovery of $\b$ and $c$ from the DN map in light of the main result of this work. To this end, we simulate equation \eqref{eq:wv} forward in time following the method of \cite{kaltenbacher2023simultaneous}, i.e., we rewrite equation \eqref{eq:wv} as
\begin{align*} 
    \begin{dcases}
        c^{-2}(1 - 2 \b u) \p_t u - \Delta \int_0^t u(r,x) \d r = 0, &\text{ in } \M, \\
        u = f, &\text{ on } \Sigma, \\
        u = \p_t u \equiv 0, &\text{ on } \M_0,
    \end{dcases}
\end{align*}
and we solve this equation numerically using a Crank-Nicolson time integration scheme. We consider the spatial domain $\Omega = [-1,1]^3$ (technically, a smoothened version of this domain to enforce that $\p \Omega$ is smooth), and we uniformly discretize in time and space, with time step $\Delta t$ and mesh size $\Delta x$, respectively. Moreover, we set $T>\text{diam}_{g}(\Omega)$.

The upper bound $\delta$ as defined in inequality \eqref{eq:DNdelta} cannot be numerically approximated. Instead, we will examine the behavior of $\norm{\p_\nu(u_1 - u_2)|_{\Sigma}}_{H^{s-2}(\Sigma)}$ for multiple boundary profiles $f \in H_{0,\eps_0}^{s+3}(\Sigma), \, s > \frac{9}{2}$, where $u_\ind, \, \ind=1,2$ is the solution to equation \eqref{eq:wv} with $c=c_\ind, \, \b = \b_\ind$ and where we denote the corresponding numerical solution by $\tilde{u}_\ind$. By applying a second-order finite difference scheme to $\tilde{u}$ at $\p \Omega$, we approximate $\p_\nu u|_{\Sigma}$ at each time step, and we denote this approximation by $D_{\nu} \tilde{u}$. Since the value for $\epsf > 0$ that guarantees well-posedness in the cases considered is unknown, we scale the boundary profiles until $\norm{D_{\nu} (\tilde{u}_1 - \tilde{u}_2)}_{H^{s-2}(\Sigma)}$ remains numerically stable. To reduce the influence of discretization errors caused by taking derivatives numerically, we will only examine $\norm{D_{\nu} (\tilde{u}_1 - \tilde{u}_2)}_{L^2(\Sigma)}$.

We will consider boundary profiles that are smooth, as this is desirable from a numerical point of view. We limit the results shown to those for $f(t,x;C) = C e^{-t^{-2}}$, with $C$ a scaling constant to enforce well-posedness. We note that the results for other boundary profiles tested are qualitatively similar, provided these do not exhibit large jumps with respect to the grid. We also note that for other boundary profiles tested, we only enforced the compatibility conditions up to a certain order.

\begin{figure}[h]
    \centering
    
    \begin{subfigure}[b]{0.495\textwidth}
        \centering
        \includegraphics[width=\textwidth]{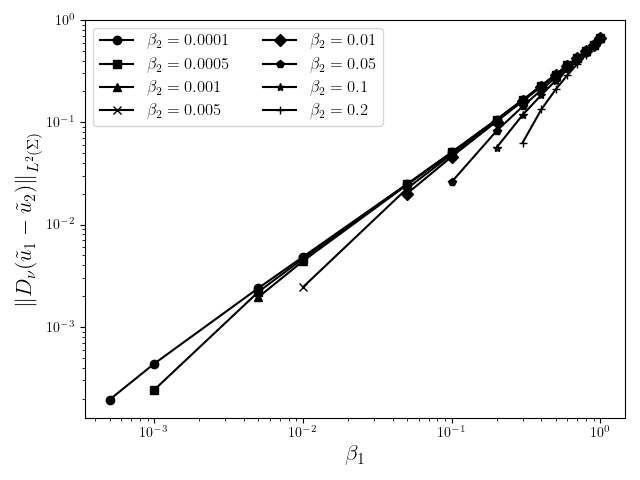}
    \end{subfigure}
    \begin{subfigure}[b]{0.495\textwidth}
        \centering
        \includegraphics[width=\textwidth]{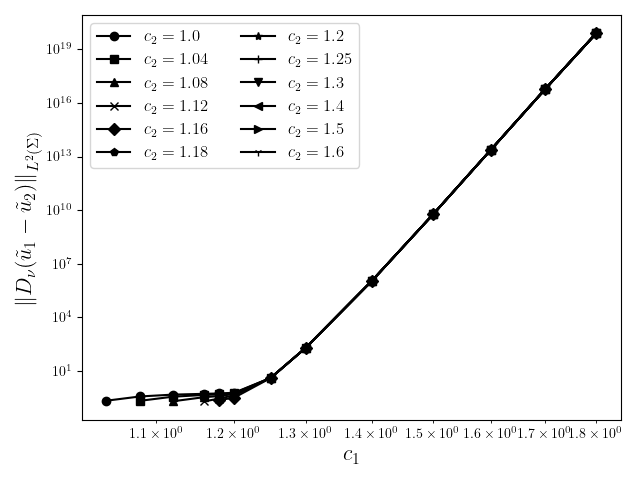}
    \end{subfigure}

    \caption{A log-log plot of $\norm{D_{\nu} (\tilde{u}_1 - \tilde{u}_2)}_{L^2(\Sigma)}$ with $f(t,x) = \frac{1}{10} e^{-t^{-2}}$ for $\b_1,\b_2 \in [10^{-4},1], \, c \equiv 1$ (left) and for $c_1,c_2 \in [1,1.8], \, \b \equiv 0$ (right), where $\Omega = [-1,1]^3, \, \Sigma = [0,3.48] \times \p \Omega, \, \Delta x = 2^{-4}, \, \Delta t = 0.03$. The results are restricted to limited values for $\b_2,c_2$ for clarity of presentation.}
    \label{fig:results_beta_c}
\end{figure}

First, we investigate the behavior of $\norm{D_{\nu} (\tilde{u}_1 - \tilde{u}_2)}_{L^2(\Sigma)}$ where we vary $\b_\ind$ and keep $c$ fixed, with $\b_\ind,c$ set to be constant for simplicity. We set $c \equiv 1$, $\Delta x = 2^{-4}$, $\Delta t=0.03$, $T = 3.48$, $C=\frac{1}{10}$, and we note that $T > \text{diam}_{g}(\Omega)$ is satisfied in this way. The results are shown in figure \ref{fig:results_beta_c} for $\b_\ind$ in the range $[10^{-4},1]$, where we only show the results for limited values $\b_2$ for clarity of presentation, and where the results are shown on a log-log plot to analyze stability. We clearly observe H\"older stability in $\b$ in this range. The H\"older stability coefficient $\mu$ seems to be dependent on $\b_1$, but this is likely to be a numerical artifact as we observe that this dependence decreases as $|\b_1 - \b_2|$ gets larger. We also note that the range in $\b$ in which H\"older stability is observed depends on the boundary profile and the corresponding scaling constant $C$. For example, stability breaks down for $|\b_1 - \b_2| > 1.1$, whereas it already breaks down for $|\b_1 - \b_2| > 0.1$ if $C = 1$. This corresponds to the requirement that $f \in H_0^{s+1}(\Sigma)$ needs to be small enough for the stability of the forward problem. These observations are also made for other $f,c$ examined.

Next, we investigate the behavior of $\norm{D_{\nu} (\tilde{u}_1 - \tilde{u}_2)}_{L^2(\Sigma)}$ while varying $c_\ind$ and keeping $\b$ fixed. We again take $c_\ind,\b$ to be constants, and set $\Delta x = 2^{-4}$, $\Delta t=0.03$, $T = 3.48$, $C=\frac{1}{10}$ for consistency with the previous study regarding stability with respect to $\b$. Moreover, for the current study we first investigate $\b \equiv 0$, as this relieves us from the ill-posedness of equation \eqref{eq:wv}, and afterwards we will investigate the case $\b > 0$. The results are shown in figure \ref{fig:results_beta_c} for $c$ in the range $[1,1.8]$, for which we note that $T>\text{diam}_{g}(\Omega)$ is satisfied, and we again limit the presented values for $c_2$ for clarity of presentation and present the results in a log-log plot to analyze stability. We observe H\"older stability in $c$ apart from for $c_1$ small, which again could be caused by numerical errors dominating in this regime as this behavior disappears when $|c_1 - c_2|$ grows large. For the case $\b > 0$, the earlier observations made for $\b \equiv 0$ still hold, only in this case the range in $c$ is limited based on the values of $\b$ and $C$, which is expected as well-posedness of equation \eqref{eq:wv} needs to be guaranteed. 

\begin{figure}[h]
    \centering
    
    \begin{subfigure}[b]{0.495\textwidth}
        \centering
        \includegraphics[width=\textwidth]{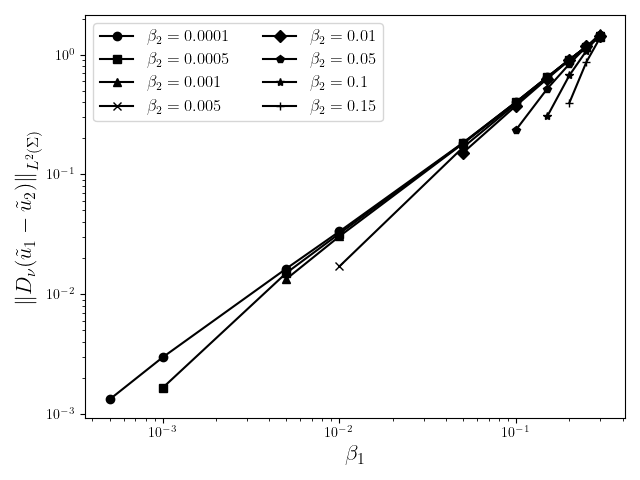}
    \end{subfigure}
    \begin{subfigure}[b]{0.495\textwidth}
        \centering
        \includegraphics[width=\textwidth]{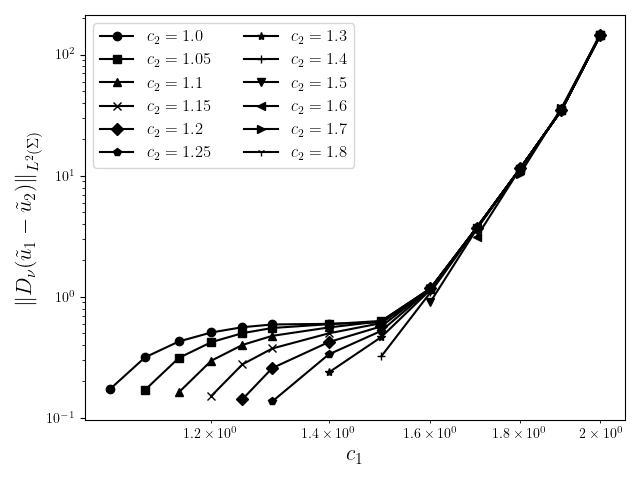}
    \end{subfigure}
    
    \caption{A log-log plot of $\norm{D_{\nu} (\tilde{u}_1 - \tilde{u}_2)}_{L^2(\Sigma)}$ with $f(t,x) = \frac{1}{10} e^{-t^{-2}}$ and $c_{\a}(x)$ as in equation \eqref{eq:c_herglotz} for $\b_1,\b_2 \in [10^{-4},0.34], \, \a = \frac{3}{2}$ (left) and for $\a_1,\a_2 \in [1,\frac{3}{2}], \, \b \equiv 0$ (right), where $\Omega = [-1,1]^3, \, \Sigma = [0,3.48] \times \p \Omega, \, \Delta x = 2^{-4}, \, \Delta t = 0.03$. The results are restricted to limited values for $\b_2,\a_2$ for clarity of presentation.}
    \label{fig:results_fol}
\end{figure}

Finally, we also numerically examine stability in the case of an isotropic radial sound speed $c(r)$ satisfying the Herglotz condition $\p_r(r c^{-1}(r)) > 0$, which allows for caustics but still satisfies the foliation condition. We set
\begin{align} \label{eq:c_herglotz}
    c_{\a}(x) = \frac{\a \sqrt{3}}{(\a - 1) \norm{x}_{2} + \sqrt{3}},
\end{align}
with $\norm{x}_{2} := \sqrt{\ip{x}{x}}$ and $\a \geq 1$ the sound speed at $x=0$, and we observe that $c$ clearly satisfies the Herglotz condition if $\a \geq 1$ and becomes the uniform sound speed $c \equiv 1$ for $\a = 1$. We again take $\Delta x = 2^{-4}$, $\Delta t=0.03$, $T = 3.48, C=\tfrac{1}{10}$ for consistency with the previous studies, where we note that $T > \text{diam}_{g}(\Omega)$ is satisfied. We examine stability in $\b$ for $\a = \frac{3}{2}$ in the range $\b \in [10^{-4},0.34]$, and similarly we examine stability in $c$ for $\b \equiv 0$ in the range $\a \in [1,\frac{3}{2}]$. The results are shown in figure \ref{fig:results_fol}, again on a log-log plot and for limited values for $\b_2,\a_2$. Again, we clearly observe H\"older stability in $\b$ and $c$ apart from seemingly numerical artifacts dependent on $\b_1,\a_1$ as discussed earlier. Moreover, the recovery of $c$ is still H\"older stable if $\b > 0$, but the range of values for $\a$ that guarantees stability decreases.


\clearpage
\appendix

\section{Proof Lemma \ref{lem:ODE_est}} \label{ap:proof_ODE}

First, recall from Picard-Lindel\"of that there exists a unique solution
\begin{align} \label{eq:ODE_sol} 
    x(t) = x(t_-) + \int_{t_-}^{t} f(s,x(s);\a(x(s))) \, \d s,
\end{align}
with $x \in \Cc^1([t_-,t_+])$. Let $f_\ind := f(t,x_\ind(t),\a_\ind(x_\ind(t))), \, \ind=1,2$ and $X := [t_-,t_+] \times \Omega \times \RR^m$. Since $f$ is globally Lipschitz in $x$ and $\a$, there exist $k_0, k_\a > 0$ such that
\begin{align} \label{eq:cond_lemma1}
    \norm{f_1(t,x_1,\a_1) - f_2(t,x_2,\a_2)} \leq k_0 \norm{x_1 - x_2} + k_\a \norm{\a_1 - \a_2}_{L^{\infty}(\Omega)},
\end{align}
for any $t \in [t_-,t_+]$, $x_\ind \in \Omega$, $\a_\ind \in L^{\infty}(\Omega)$. Letting $d(t) := \norm{x_1(t) - x_2(t)}$, we thus get from equation \eqref{eq:ODE_sol} and inequality \eqref{eq:cond_lemma1} that
\begin{align*}
    d(t) \leq d(t_-) + k_0 \int_{t_-}^t d(s) \, \d s + k_\a (t-t_-) \norm{\a_1 - \a_2}_{L^\infty(\Omega)}.
\end{align*}
Estimate \eqref{eq:est_ODE_lipschitz} then follows from Gr\"onwall's inequality, which gives that
\begin{align*}
    d(t) \leq d(t_-) e^{k_0(t-t_-)} + \frac{k_\a}{k_0} \left(e^{k_0 (t-t_-)} - 1 \right) \norm{\a_1 - \a_2}_{L^\infty(\Omega)}.
\end{align*}

For $k \geq 1$, we immediately observe from equation \eqref{eq:ODE_sol} that $x \in \Cc^k([t_-,t_+])$. We first prove estimate \eqref{eq:est_ODE_k} for $k=1$. From estimate \eqref{eq:est_ODE_lipschitz} and the additional assumptions posed for $k=1$, we get that
\begin{align*}
    \norm{x_1 - x_2}_{\Cc([t_-,t_+])} + \norm{\dot{x}_1 - \dot{x}_2}_{\Cc([t_-,t_+])} &= \norm{x_1 - x_2}_{\Cc([t_-,t_+])} + \norm{f_1 - f_2}_{\Cc(X)} \\
    &\lesssim \norm{x_1 - x_2}_{\Cc([t_-,t_+])} + \norm{\a_1 - \a_2}_{\Cc(\Omega)} \\
    &\lesssim d(t_-) + \norm{\a_1 - \a_2}_{\Cc(\Omega)}
\end{align*}
proving the $k=1$ case. 
The proof for $k \geq 2$ follows by induction, and we will illustrate the induction step by giving the proof for $k=2$. Since $\ddot{x} = \p_t f + \dot{x} \p_x f + (\dot{x} \p_x \a) \p_\a f$, we get that
\begin{align*}
    \norm{\ddot{x}_1 - \ddot{x}_2}_{\Cc([t_-,t_+])} 
    &\lesssim \norm{f_1 - f_2}_{\Cc^1(X)} + \norm{x_1 - x_2}_{\Cc^1([t_-,t_+])} + \norm{\a_1 - \a_2}_{\Cc^1(\Omega)} \\
    &\lesssim \norm{x_1 - x_2}_{\Cc^1([t_-,t_+])} + \norm{\a_1 - \a_2}_{\Cc^1(\Omega)} \\
    &\lesssim d(t_-) + \norm{\a_1 - \a_2}_{\Cc^1(\Omega)},
\end{align*}
from the estimate for $k=1$ and the additional assumptions posed for $k=2$.

\end{document}